\documentclass[11pt,leqno]{amsart}

\usepackage[all,cmtip]{xy}
\usepackage{a4,latexsym}
\usepackage[utf8]{inputenc}
\usepackage{amsfonts}
\usepackage[hidelinks]{hyperref}
\usepackage{amsxtra,amsmath,amsfonts,amscd, amssymb, mathrsfs, amsthm}
\usepackage{color, mathtools}
\usepackage{bbm}
\usepackage{enumitem}
\usepackage{xcolor}
\usepackage{graphicx}
\usepackage{tikz-cd}
\usepackage{braket}
\usepackage[all]{xypic}
\usepackage[toc,page,title,titletoc,header]{appendix}
\usepackage[capitalize]{cleveref}

\numberwithin{paragraph}{section}

\setlist[enumerate]{label=\it{(\roman*)},
	ref=\it{(\roman*)}}

\newcommand{\C}{{\mathbb C}}
\newcommand{\R}{{\mathbb R}}

\newtheorem{theorem}{Theorem}[section]

\newtheorem{corollary}[theorem]{Corollary}

\newtheorem{lemma}[theorem]{Lemma}
\newtheorem{lemma*}{Lemma}

\theoremstyle{definition}

\newtheorem{proposition&definition}[theorem]{Proposition\&Definition}
\newtheorem{lemma&definition}[theorem]{Lemma\&Definition}
\newtheorem{theorem&definition}[theorem]{Theorem\&Definition}

\newtheorem{example*}{Example}
\newtheorem{remark}{Remark}

\newtheorem{question*}{Question}

\newtheorem*{theorem2}{Theorem}

\def\vn{{\vec n}}
\def\p{{\partial}}
\def\S{\mathbf S}

\def\U{\mathcal U}
\def\W{\mathcal W}

\numberwithin{equation}{section}

\begin{document}
	
	\title[ $L^p$-boundedness  of multi-parameter Fourier integral operators]{ $L^p$-boundedness  of multi-parameter Fourier integral operators}
	
	\author[J.~Cheng]{Jinhua Cheng}
	\address{J. Cheng,  Department of Mathematics, Westlake University, 310024 Hangzhou, P. R. China }
	\email{chengjinhua@westlake.edu.cn}
	

	\begin{abstract}
	We study a specific class of Fourier integral operators characterized by symbols belonging to the multi-parameter Hörmander class $\mathbf{S}^m(\R^{ n_1} \times \R^{ n_2}  \times \cdots \times \R^{n_d}  )$, where $n= n_1 + n_2 +\cdots + n_d$. Our investigation focuses on cases where the phase function $\Phi(x,\xi)$ can be decomposed into a sum of individual components $\Phi_i(x_i,\xi_i)$, with each component satisfying a non-degeneracy condition.  We extend the Seeger-Sogge-Stein theorem under the condition that the dimension $ n_i \ge 2$ for each
	$1\le i \le d$. As a corollary, we obtain the boundedness of multi-parameter Fourier integral operators  on local Hardy spaces, Lipschitz spaces, and Sobolev spaces.
	\end{abstract}
	
	\keywords{Multi-parameter Fourier integral operators, Hardy space, Cone decomposition, Littlewood-Paley decomposition} 
	\subjclass{{Primary 14G40; Secondary 11G10, 14G22}}
	
	\maketitle
	
	\setcounter{tocdepth}{1}
	

\section{Introduction}
\setcounter{equation}{0}
This paper investigates the mapping properties of Fourier integral operators.   As  our considerations will be local, we will work on $\R^n$ for some $n \ge 2$. 
Let $f$  be a Schwartz function,  a  Fourier integral operator, as defined by Hörmander \cite{H71}, is of the form
\begin{equation}\label{linear FIO}
 Tf(x) = \int_{\R^n}  e^{2\pi  i  \Phi(x ,\xi)  }  \sigma(x, \xi) \widehat{f}(\xi)  d \xi,
\end{equation}
 where $\widehat{f}(\xi) $ is  the Fourier transform $ \widehat{f}(\xi ) = \int_{\R^n} e ^{- 2 \pi \mathbf{i} x \cdot \xi } f(x) dx$, of $f$. 
  One  requires the symbol 
$\sigma(x, \xi)\in \mathcal{C}^\infty(\mathbb{R}^n\times \R^n )$ has   compact support in $x$.  A symbol $\sigma$ is said to  belong to the H\"ormander class $\mathbf{S}^m(\mathbb{R}^n)$ if it 
 satisfies the estimates
\begin{equation} \label{Hormander class}
\left| \partial_\xi^\alpha \partial_{x }^\beta \sigma(x,   \xi) \right  |\le C_{\alpha,\beta}(1+|\xi|)^{m-|\alpha|}
\end{equation}
for all  multi-indices $\alpha,\beta$. The phase function $\Phi(x,\xi)\in  \mathcal{C}^{\infty}( \mathbb{R}^n \times (\mathbb{R}^n\setminus\{0\})   )$ is a real-valued, positively homogeneous of degree one in the variable $\xi$, 
what's more, $\Phi$ obeys  the \emph{non-degenercy} condition
\begin{equation}\label{nondegeneracy}
\det\left(  \frac{\partial^2 \Phi(x,\xi)}{\partial x_i\partial \xi_j} \right) \neq 0,~~~~~\xi\in \mathbb{R}^n\setminus\{0\}.
\end{equation}

Fourier integral operators defined as (\ref{linear FIO})  -  (\ref{nondegeneracy})  have been studied extensively and intensively  over the past decades due to their wide applicability in partial differential equations.
 Eskin \cite{E70} and H\"ormander \cite{H71} showed that  $T$  of order $0$ is bounded from $L^2$ to itself,
however,  for $p\neq 2$, Fourier integral operators may not  be bounded on $L^p$. The  $L^p-$estimate for Fourier integral operators was first investigated by Duistermaat and H\"ormander \cite{DH72}, Colin de Verdi\'ere and Frisch \cite{CF76}, Brenner \cite{B77}, 
 then the sharpness of the order $-(n-1)|1/2-1/p|$ was shown by  Peral \cite{P80} and  Miyachi \cite{M80}, Beals \cite{B82}, and eventually the optimal result was obtained by Seeger, Sogge and Stein \cite{SSS91} in 1991, where they showed the operator $T$ of order $m$ is bounded on $L^p(\mathbb{R}^n)$ if $m\le -(n-1)|1/p-1/2|$~ for $~1<p<\infty$.  Due to the atomic representation of the  Hardy space $H^1(\mathbb{R}^n)$ and the complex interpolation theorem obtained  by Fefferman and Stein, see \cite{F71} and \cite{FS72},  they essentially proved that $T$ of order $-(n-1)/2$ maps the Hardy space $H^1(\mathbb{R}^n)$ into $L^1(\mathbb{R}^n)$.

\begin{remark}
One can not expect $T$ of order $-(n-1)/2$ to be bounded on $L^1(\R^n)$,  however, $T$ is of weak-type (1,1),  for more details, see Tao \cite{T04}.
\end{remark}

 On the other hand,  the multi-parameter theory, sometimes called product theory corresponds to range of questions which are concerned with issues of harmonic analysis that are invariant with respect to a family of dilations $\delta: x\rightarrow \delta x=(\delta_1 x_1, \dots, \delta_d x_d), ~\delta_i >0, i=1,\dots,d$.  The multi- parameter function spaces and boundedness of Fourier multipliers, singular integral operators on such spaces have been extensively studied by many authors, for instance, M\"uller, Ricci and Stein \cite{MRS95, MRS96}, R. Fefferman \cite{F85, F86, F87},  R. Fefferman and Pipher \cite{FP05, FP97} , Muscalu, Pipher,  Tao and Thiele  \cite{MPTT04, MPTT06}.

Wang \cite{W22}  considered the product H\"ormander class $\mathbf{S}^m(\R\times \R \times \cdots \times \R )$, a symbol  $\sigma$   $\in \mathbf{S}^m(\R\times \R \times \cdots \times \R )$ if 
\begin{equation}\label{product class}
\left|\p_\xi^\alpha\p_{x}^\beta \sigma(x, \xi)\right|~\leq~C_{\alpha, \beta}\left(1+|\xi|\right)^m \prod_{i=1}^n \left({1\over 1+|\xi_i|}\right)^{|\alpha_i|}
\end{equation}
for every multi-indices $\alpha,\beta$.  Wang proved the following Theorem
\begin{theorem2}[Wang, 2022]\label{Wang}
Let $T$ be  defined as (\ref{linear FIO}),   (\ref{nondegeneracy}) and (\ref{product class}). Suppose
$\sigma \in  \mathbf{S}^m(\R\times \R \times \cdots \times \R )$ for $-(n-1)/2<m\le 0$. We have
\begin{equation*}
\left\| Tf \right \|_{L^p(\R^n)} \le C  \left\| f \right \|_{L^p(\R^n)}, \qquad 1<p<\infty
\end{equation*}
whenever
 \begin{equation*}
m \le -(n-1) \left| \frac{1}{2} - \frac{1}{ p} \right|.
\end{equation*}
\end{theorem2}

Let  $\xi=(\xi_1, \xi_2, \dots, \xi_d)\in \R^\vn$ and  $x=(x_1, x_2, \dots, x_d)\in \R^\vn$, where
\[
\R^\vn= \R^{n_1}\times \R^{n_2} \times \cdots \times \R^{n_d}, \quad n=n_1+n_2+\cdots+n_d, \quad d\ge 2.
\]
We say that $\sigma$ belongs to the product H\"ormander class  $\mathbf{S}^m_{\rho, \delta}(\R^{\vn} )  $,  if for all multi-indices
$\alpha,\beta$, it satisfies
\begin{equation}\label{product symbol}
\left|\p_\xi^\alpha\p_{x}^\beta \sigma(x,\xi)\right| \le C_{\alpha, \beta}\left(1+|\xi|\right)^m \prod_{i=1}^d (1+|\xi_i|)^{ -\rho |\alpha_i|+\delta |\beta_i|}.
\end{equation}
We denote the standard product H\"ormander class by  $\mathbf{S}^m(\R^\vn)=\mathbf{S}^m_{1, 0}(\R^\vn ) $.
Moreover,  we restrict the phase function to be the particular form, that is,
\begin{equation}\label{phase function}
\Phi(x,\xi)=\sum_{i=1}^d \Phi_i(x_i,\xi_i).
\end{equation}

Our main result is stated as below
\begin{theorem}\label{classical Hardy space}
Let the Fourier integral operator $T$ be defined as (\ref{linear FIO}) and (\ref{phase function}).
 Each $\Phi_i(x_i,\xi_i)\in \mathcal{C}^{\infty}( \mathbb{R}^{n_i} \times (\mathbb{R}^{n_i}\setminus\{0\})   )$ is a real-valued, positively homogeneous of degree one in the variable $\xi_i$, and satisfies the \emph{non-degenercy} condition (\ref{nondegeneracy}).
  Suppose    $\sigma \in \mathbf{S}^m(\R^\vn)$ and  $n_i \ge 2$ for $i=1,2,\dots,d$,   then we have 
\begin{equation}\label{aim}
 \left\| T  f \right \|_{L^p(\mathbb{R}^n)} \le  C  \left\| f \right \|_{L^p(\mathbb{R}^n)} , \quad  1<p < \infty
\end{equation}
whenever
\[
m\le - (n-d) \left |\frac{1}{2}-\frac{1}{p} \right|.
\]
\end{theorem}

\begin{remark}
 When $d=1$, Theorem \ref{classical Hardy space} reduces to the classical result of Seeger, Sogge, and Stein. For more details on Fourier integral operators, one can refer to the monographs by Stein \cite{S93} and Sogge \cite{Sogge93}.
\end{remark}

\begin{remark}
Note that   $(n-1)$  is replaced by $(n-d)$. This adjustment is due to the assumption that each $\Phi_i$ satisfies the non-degeneracy condition.  For instance, set $\Phi_i (x_i, \xi_i)= x_i \cdot \xi_i + |\xi_i| $ for each $1 \le i \le d$, then the rank of the Hessian matrix $(\Phi_i)_{ \xi_i\xi_i}$ is $n_i -1 $. Consequently  the total  rank of the Hessian matrix is $\sum_{i=1}^d (n_i-1) = n-d$. Therefore, according to  the classical results of Seeger, Sogge and Stein in  \cite{SSS91}, it is expected that $n-1$ should be replaced by  $n-d$.
\end{remark}

\begin{remark}
Throughout the paper, the constant $C$ may be dependent on the $\sigma,\Phi, \alpha, \beta$, $d, p, n$, but not dependent on $f$ or other specific function and  may be different line by line.
\end{remark}

\begin{corollary}\label{Lip}
 Let  $T$ be the Fourier integral operator in Theorem \ref{classical Hardy space}. For each 
 $i=1,2, \dots, d$,  with $n_i \ge 2$,   we have
 \begin{equation}
\left \| T f  \right\|_{Lip(\alpha)} \le C  \left\| f \right\|_{Lip(\alpha)}, \quad \sigma \in \S^{- \frac{n-d}{2}}(\R^\vn ),
\end{equation}
Here,  $Lip(\alpha)$  denotes the Lipschitz space of order  $\alpha > 0$.
\end{corollary}

\begin{corollary}\label{h^1}
 Let  $T$ be the Fourier integral operator in Theorem \ref{classical Hardy space}. For each 
 $i=1,2, \dots, d$,  with $n_i \ge 2$,   we have
\begin{equation}
\left \| T f  \right\|_{h^1(\R^n)} \le C  \left\| f \right\|_{ h^1(\R^n)}, \quad \sigma \in \S^{- \frac{n-d}{2}}(\R^\vn ).
\end{equation}
Here,  $h^1(\R^n)$  denotes the  local Hardy  space.  
\end{corollary}

\begin{corollary}\label{Sobolev}
 Let  $T$ be the Fourier integral operator in Theorem \ref{classical Hardy space}. For each 
 $i=1,2, \dots, d$,  with $n_i \ge 2$,  suppose  $\sigma \in \S^{m}(\R^\vn ) $,  we have
\begin{equation}
\left \| T f  \right\|_{L^p_s(\R^n)} \le C  \left\| f \right\|_{L^p_s(\R^n)}, \quad 1<p< \infty,
\end{equation}
whenever
\[
m \le -(n-d)\left | \frac{1}{p} -  \frac{1}{2}  \right |.
\]
Here,  $L^p_s(\R^n)$  is  the   Sobolev  space.  
\end{corollary}

The paper is organized as follows.  In section $2$, we shall introduce the necessary notation and make a cone decomposition of $T$. In section $3$, we prove the $L^2$ boundedness of $T$.  In section $4$,  we will further decompose the operator $T$ by Seeger-Sogge-Stein decomposition. 
Section $5$ will be devoted of the majorization of the kernels of $T$    and in the last section shall discuss further $  L^p$-boundedness of more general operator $T$.


\section{Preliminaries}
\setcounter{equation}{0}
It is  well-known  that $T$  of order $0$ is bounded from $L^2$ to itself, and for completeness we will prove this fact in the next section. 
Therefore,  by complex interpolation theorem,  in order to prove Theorem \ref{classical Hardy space}, it suffices to show
 $T$  of order $-(n-d)/2$ and its adjoint operator $T^*$ are both 
  bounded from $H^1(\mathbb{R}^n ) $ to $L^1(\R^n) $.
 Due to the atomic decomposition of the Hardy space $H^1(\R^n)$,  the problem then is reduced to show
\begin{equation}
\int_{\R^n} \left |Ta(x) \right | dx \le  C, \quad  \sigma \in \mathbf{S}^{-(n-d)/2}(\R^\vn),
\end{equation}
and 
\begin{equation}
\int_{\R^n} \left |T^*a(x) \right | dx \le  C, \quad  \sigma \in \mathbf{S}^{-(n-d)/2}(\R^\vn).
\end{equation}
 Here  $a(x)$ is a $H^1(\R^n) $ atom, that is, $a(x)$ is supported on a ball $B$ such that $||a||_{L^\infty}\le |B|^{-1} $ and $\int_{\R^n} a(x)dx=0$.
Now fix an atom $a (x)$ supported  on  a $B_r(\overline{y}) \subset \mathbb{R}^n$ with radius $r$, we can actually assume that    $r<1$, otherwise, the estimate is trivial.  Because of our assumption the symbol $\sigma(x,\xi)$ has a compact support in $x$-variable, we have
\begin{equation}\label{trivial estimate}
\left\|T  a \right\|_{L^1}\le \left\|T  a \right\|_{L^2}\le C \left \| a \right \|_{L^2}\le C r^{-n/2}\le  C.
\end{equation}
The first inequality holds because $T $ has fixed compact support and Cauchy-Schwarz inequality; the second follows from the well-known  $L^2$-boundedness of $T $ of order $0$.  In the following, we primarily focus on the operator  $T$;   In most cases,   the same argument is applicable to $T^*$,  and any differences will be pointed out as needed.

We will employ a similar approach to the one used for one-parameter Fourier integral operators to prove the theorem for multi-parameter Fourier integral operators. Specifically, we will divide $\R^n$  into two parts: the region of influence and its complement. For the region of influence, we will use the $L^2$ estimates of $T$.  For the region outside of influence, we  need critical control over the operator kernel. However, multi-parameter Fourier integral operators present different challenges. For example, in the dyadic ring  $\{ \xi : 2^{ j-1} \le |\xi| \le 2^{j+1}\}, ~j>0$,
in the one-parameter case, we have  $ |\p_\xi \sigma(x,\xi)| \le C  2^{-j} $,  but this inequality does not hold in the multi-parameter case. To precisely control the magnitude of $|\xi_i|$  in each subspace, we will perform a multi-parameter Littlewood-Paley decomposition and introduce cone decomposition.

\subsection{Cone decomposition for $\S^m(\R^\vn )$}

By examining the differential inequality of $\sigma(x, \xi)$, that is,
\begin{equation*}
\left|\p_\xi^\alpha\p_{x}^\beta \sigma(x,\xi)\right|~\leq~C_{\alpha\beta}\left(1+|\xi|\right)^m \prod_{i=1}^d \left({1\over 1+|\xi_i|}\right)^{|\alpha_i|}
\end{equation*}
for any multi-indices $\alpha, \beta$, the differential inequality holds. We can partition the $\xi$-space into different regions, where in each region we know the size of each $|\xi_i|$ and the magnitude of $|\xi|$ is approximately $\sim 2^j$. Therefore, it is natural to consider the following multi-parameter Littlewood-Paley decomposition.  Let $\varphi$ be a smooth function on $\mathbb{R}$ that satisfies
 \begin{equation}\label{varphi}
\varphi(t)=1 \quad \textit{if } \quad |t|\le 1; \quad  \varphi(t)=0, \quad \textit{if}  \quad |t|\ge 2.
\end{equation}
Define 
\begin{equation}\label{phi_j}
\begin{array}{lc} \displaystyle
\phi_{j}(\xi) =\varphi(2^{-j}|\xi|) -\varphi(2^{-j+1}|\xi|), \quad j\in \mathbb{Z}, j>0;
 \quad 
\phi_{0} (\xi) = \varphi(|\xi|).
\end{array}
\end{equation}
and

\begin{equation}
\phi_{j-\ell_i}(\xi_i) =\varphi(2^{-j+\ell_i}|\xi_i|) -\varphi(2^{-j+\ell_i+1}|\xi_i|) ,\quad j\in \mathbb{Z}, \quad j\ge \ell_i,
\end{equation}
\[
\phi_{ j \ell}(\xi)= \prod_{i=1}^d \phi_{j-\ell_i}(\xi_i).
\]
Here, $\xi = (\xi_1, \xi_2, \dots, \xi_d) \in \mathbb{R}^{n_1} \times \mathbb{R}^{n_2} \times \cdots \times \mathbb{R}^{n_d}$, and $\ell = (\ell_1, \ell_2, \dots, \ell_d)$. Due to symmetry, we can assume without loss of generality that each $\ell_i$ is a non-negative integer, with $\ell_d = 0$. Consequently, we have $|\xi_i| \sim 2^{j - \ell_i}$ and $|\xi_d| \sim |\xi| \sim 2^j$.

Note that the support of $\phi_{j\ell}(\xi)$ is contained within the following set, depicted in the red region of Figure \ref{fig:example8}.

\begin{equation}
\left \{\xi \in \R^\vn : 2^{j - \ell_i-1} \le |\xi_i|\le 2^{j-\ell_i+1},~~1 \le i \le d \right  \}.
\end{equation}

\begin{figure}[htbp]
  \centering
  \includegraphics[width=7cm]{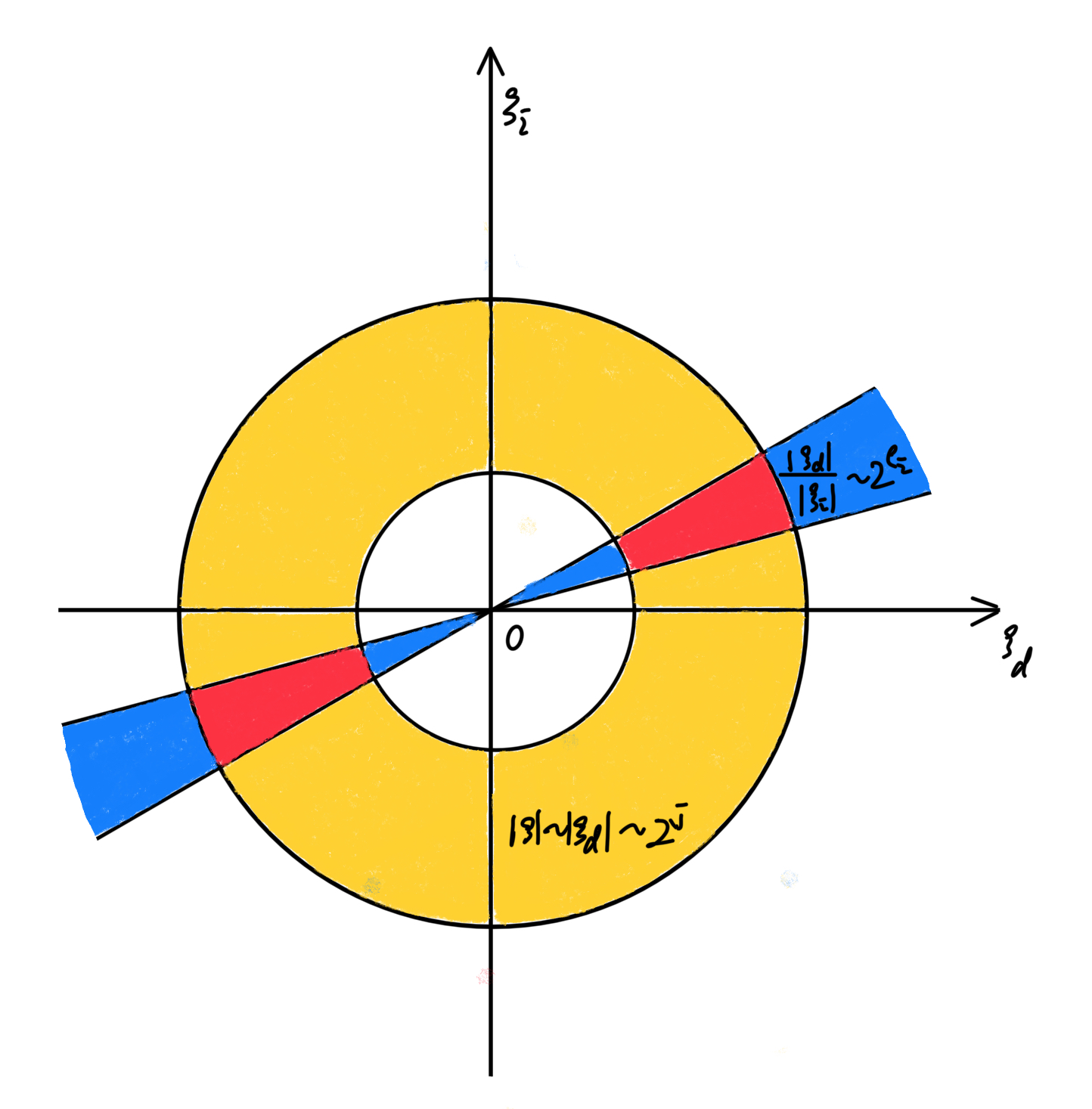}
  \caption{The support of $\phi_{j \ell}(\xi) $ }
  \label{fig:example8}
\end{figure}
 
And we have 
\[
\left |\p_\xi^\alpha \phi_{j \ell}(\xi) \right  |\le  C_\alpha \prod_{ i =1}^d  2^{- (j -\ell_i) |\alpha_i|}, 
\]
Now define partial operators
\begin{equation}\label{partial operator jell}
T_{j \ell }f(x)=\int_{\mathbb{R}^n} f(y)K_{j \ell  }(x, y)dy, ~~   K_{j \ell   }(x, y)=\int_{\mathbb{R}^n} e^{2\pi i ( \Phi(x,\xi)-y\cdot \xi ) }\sigma(x,\xi) \phi_{j \ell} (\xi) d\xi.
\end{equation}
Then, we have the "partition of unity",
\begin{equation*}
 \sum_{j\ge 0 } \prod_{i=1 }^{d-1}\sum_{0 \le \ell_i \le j} \phi_{j \ell} (\xi)=
 \prod_{i =1}^{d-1} \sum_{\ell_i \ge 0} \sum_{j \ge \ell_M} \phi_{j \ell}(\xi) =   \prod_{i =1}^{d-1} \sum_{\ell_i \ge 0}  \delta_\ell(\xi),
\end{equation*}
where
\[
\ell_M = \max_{1 \le i \le d} \ell_i, \quad \delta_\ell(\xi) = \sum_{j \ge \ell_M} \phi_{j \ell}(\xi).
\]

Note that the support of $\delta_{\ell}(\xi)$ is roughly contained within a cone, which is why the above decomposition is referred to as a cone decomposition. This is illustrated in the blue region of Figure \ref{fig:example9}.

\begin{figure}[htbp]
  \centering
  \includegraphics[width=7cm]{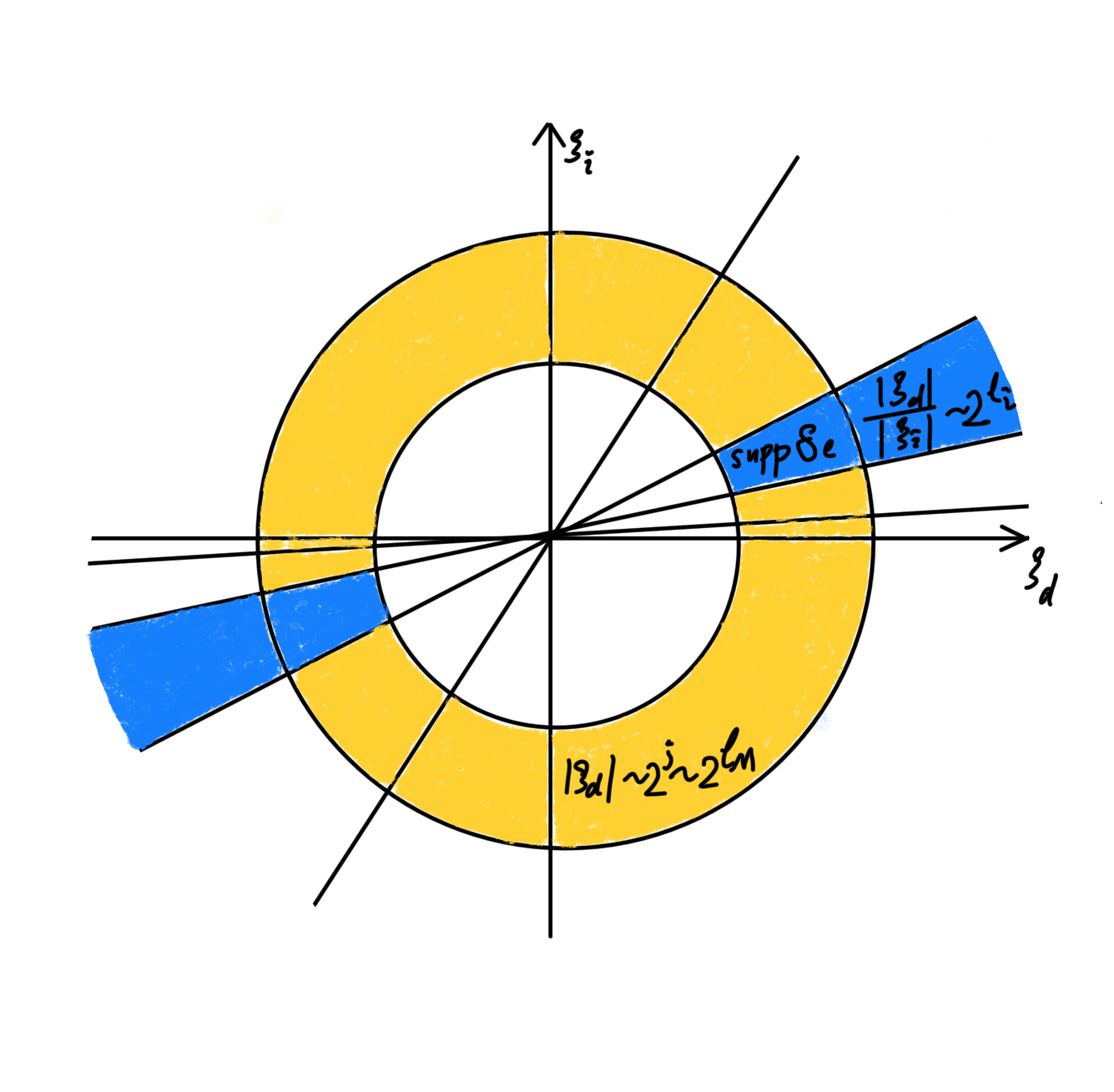}
  \caption{ The support of  $\delta_{\ell}(\xi) $  }
   \label{fig:example9}
\end{figure}

\begin{equation}\label{partial operator ell}
T_{ \ell }f(x)=\int_{\mathbb{R}^n} f(y)K_{ \ell  }(x, y)dy, \quad   K_{ \ell   }(x, y)=\int_{\mathbb{R}^n} e^{2\pi i ( \Phi(x,\xi)-y\cdot \xi ) }\sigma(x,\xi) \delta_{\ell} (\xi) d\xi.
\end{equation}
Thus, it suffices to show that for any $H^1(\R^n)$ atom $a$, the following estimate holds.
\[
\left \|  \sum_{j\ge 0 } \prod_{i=1 }^{d-1}\sum_{0 \le \ell_i \le j} T_{j \ell}  a \right \|_{L^1(\R^n)} \le C,
\quad \text{or} \quad 
\left \| \prod_{i =1}^{d-1} \sum_{\ell_i \ge 0}  T_{ \ell}  a \right \|_{L^1(\R^n)} \le C.
\]

\section{$L^2$ estimates }
In this section, we  prove the following lemma concerning $L^2$-boundedness of $T$.
\begin{lemma}\label{L^2 estimates lemma}
Let $T$ be defined as in Theorem \ref{classical Hardy space}, $T_\ell$ be defined as in (\ref{partial operator ell}), and $T_\ell^*$ be the adjoint operator of $T_\ell$. Then 
\[
\left\|T f \right\|_{L^2(\R^n )  } \le C  \left\|f  \right \|_{L^2(\R^n)}, \qquad \sigma \in \S^0(\R^\vn).
\]
Moreover, if  $\sigma \in \S^{-(n-d)/2}(\R^{\vn})$, then we have
\begin{equation}\label{L^p to L^2}
\begin{array}{lc}\displaystyle
 \left\|  T_\ell f \right\|_{L^2(\R^n)}
   \le C \prod_{i =1}^{d-1}  2^{-\ell_i n_i \frac{n-d}{2n} }  \left\|  f \right\|_{L^p(\R^n)},
   \quad  \frac{1}{p}=\frac{1}{2} + \frac{n-d }{2n}.
   \end{array}
\end{equation}
\begin{equation}\label{L^2 to L^q}
\begin{array}{lc}\displaystyle
~~~~~~~ \left\|  T^*_\ell f \right\|_{L^2(\R^n)}
   \le C \prod_{i =1}^{d-1}  2^{-\ell_i n_i \frac{n-d}{2n} }  \left\|  f \right\|_{L^p(\R^n)},
   \quad  \frac{1}{p}=\frac{1}{2} + \frac{n-d }{2n}.
   \end{array}
\end{equation}
\end{lemma}

\begin{proof}
By Plancherel's theorem, matters are then reduced to a similar assertion for the operator $S$,
\[
Sf (x) =\int_{\R^n } e^{2\pi i \Phi(x, \xi) } \sigma(x, \xi ) f(\xi) d\xi.
\]
whose adjoint operator is 
\[
S^*f (\xi)=\int_{\mathbb{R}^n} e^{-2\pi i\Phi(x,\xi)} \overline{\sigma}(x,\xi) f(x)dx.
\]
We aim to show $S^*S$ is bounded on $L^2(\mathbb{R}^n)$, we write
\[
S^*S f(\xi)= \int_{\mathbb{R}^n}f(\eta) K (\xi,\eta) d\eta,
\]
where
\[
K(\xi,\eta)=\int_{\mathbb{R}^n} e^{2\pi \mathbf{i}  \left (\Phi(x,\eta)-\Phi(x,\xi) \right )} \sigma(x,\eta)\overline{\sigma}(x,\xi) dx.
\]
Let $\mathbf{c}$ be a small positive constant.  We define an {\itshape narrow cone} as follows: suppose $\xi$ and $\eta$ belong to a same narrow cone and $|\eta|\leq|\xi|$. By writing $\eta=\rho\xi+\eta^\dagger$ for $0\leq\rho\leq1$ and $\eta^\dagger$ perpendicular to $\xi$, we require  $|\eta^\dagger|\leq \mathbf{c}\rho|\xi|$. The value of $\mathbf{c}$ depends on $\Phi$.

Clearly,  we can decompose the frequency space for which $S$ or $S^*$  can be written as a finite sum of  partial operators.  Each one of them has  a symbol  supported in such a narrow cone.
Recall the estimate given at { 3.1.1}, chapter IX of Stein \cite{S93}.
We have
\begin{equation}\label{Phi property}
    \Big|\nabla_x \Big(\Phi(x,\xi)-\Phi(x,\eta)\Big)\Big|~\ge~C_\Phi~|\xi-\eta|,
\end{equation}
whenever $\xi$ and $\eta$ belong to a same narrow cone.  
Recall that $\sigma(x,\xi)$ has a compact support in $x$. Hence that $K(\xi,\eta)$ is bounded in norm. By using (\ref{Phi property}), 
an $N$-fold integration by parts  $w.r.t~x$ gives
\begin{equation}
    \begin{array}{lr}\displaystyle
\Big|K (\xi,\eta)\Big|~\leq~C_N ~
|\xi-\eta|^{-N}
\Big|\int_{\mathbb{R}^{n}} e^{2\pi  i \left(\Phi(x,\eta)-\Phi(x,\xi)\right)} \nabla_x^N\Big(\sigma(x,\eta)\bar{\sigma}(x,\xi)\Big)dx\Big|
\end{array}
\end{equation}
for $\xi\neq\eta$.  Together with (\ref{product class}), we have
\begin{equation}\label{S sharp estimate}
 \Big| K (\xi,\eta)\Big|~\leq~C_N ~\Big({\frac{1}{1+|\xi-\eta|}}\Big)^{N}
\end{equation}
for every $N\ge1$. To conclude the $L^2$-boundedness of $T$, we write
\begin{equation}
    \begin{array}{lr}\displaystyle  
\Big\| S^* Sf\Big\|_{L^2(\mathbb{R}^{n})}~=~\left \{\int_{\mathbb{R}^{n}}\left |\int_{\mathbb{R}^{n}} f(\eta) K (\xi,\eta)d\eta \right|^2d\xi\right \}^{\frac{1}{2}}
\\\\ \displaystyle
~=~\left \{\int_{\mathbb{R}^{n}}  \left|\int_{\mathbb{R}^{n}} f(\xi-\zeta) K (\xi,\xi-\zeta)d\zeta \right|^2d\xi\right\}^{\frac{1}{2}}\qquad (~ \zeta=\xi-\eta ~) 
\\\\ \displaystyle
~\leq~~\int_{\mathbb{R}^{n}} \left \{\int_{\mathbb{R}^{n}} \left |f(\xi-\zeta)\right |^2 \right | K (\xi,\xi-\zeta)\Big|^2 d\xi \Big\}^{\frac{1}{2}}d\zeta
\\\\ \displaystyle~~\leq~C_N~\int_{\mathbb{R}^{n}} \left \{\int_{\mathbb{R}^{n}} \left | f(\xi-\zeta)\Big|^2 \right ({\frac{1}{1+|\zeta|}}\Big)^{2N} d\xi \right \}^{\frac{1}{2} }d\zeta\qquad \hbox{\small{by (\ref{S sharp estimate})}}
\\\\ \displaystyle
~=~C_N ~\left \| f\right \|_{L^2(\mathbb{R}^{n})}\int_{\mathbb{R}^{n}}  \left ({\frac{1}{1+|\zeta|}}\right )^{N}d\zeta
\\\\ \displaystyle
~\leq~C~\left \| f\right \|_{L^2(\mathbb{R}^{n})},\qquad \hbox{\small{for $N$  sufficiently large. }}
\end{array}
\end{equation}



Now we turn to the proof of estimate (\ref{L^p to L^2}). Note that we can write  $T_{\ell} f(x)  $ as 
\[
\begin{array}{lc}\displaystyle
T_{\ell} f(x)  =\int_{\mathbb{R}^n } e^{2\pi  i \Phi(x,\xi) }  \sigma(x,\xi) \delta_\ell(\xi)  \widehat{f}(\xi )d\xi
\\\\ \displaystyle
=  \int_{\mathbb{R}^n } e^{2\pi i  \Phi(x,\xi) }  \left[  \sigma(x,\xi) (1+ |\xi|^2)^{ \frac{n-d}{4} }   \right]  \widehat{ \mathcal{T}_\ell f}(\xi) d\xi,
\end{array}
\]
where 
\[
\widehat{ \mathcal{T}_\ell f}(\xi)  =  (1+ |\xi|^2)^{- \frac{n-d}{4} }  \delta_\ell(\xi)  \widehat{f}(\xi ).
\]
Note that 
 \[
 \sigma(x,\xi) (1+ |\xi|^2)^{ \frac{n-d}{4} } \in \S^0(\R^\vn), 
 \]
therefore, by applying the $0$-th order multi-parameter Fourier integral operator's $L^2$-boundedness and the Plancherel theorem, we only need to prove that
\[
  \left\|  \mathcal{T}_\ell f \right\|_{L^2(\R^n)}
   \le C \prod_{i =1}^{d-1}  2^{-\ell_i n_i \frac{n-d}{2n} }  \left\|  f \right\|_{L^p(\R^n)},
   \quad  \frac{1}{p}=\frac{1}{2} + \frac{n-d }{2n}.
\]
Now, rewrite $\mathcal{T}_\ell$ as
\[
\mathcal{T}_{\ell} f (x) = \int_{\R^n} f(y) \mathcal{R}_\ell(x-y) dy,
\]
\[
 \mathcal{R}_\ell(x) = \int_{\R^n} e^{2\pi \i x \cdot \xi }  (1+ |\xi|^2)^{- \frac{n-d}{4} }  \delta_\ell(\xi)  d\xi.
\]
We first prove 
\[
\left| \mathcal{R}_\ell (x ) \right| 
\le C \prod_{i =1}^{d-1}    2^{-\ell_i n_i \frac{n-d}{2n} } 
|x_i|^{-n_i(1-(n-d)/2n)}.
\]
We consider 
\[
 \mathcal{R}_{j \ell}(x) = \int_{\R^n} e^{2\pi \i x \cdot \xi }  (1+ |\xi|^2)^{- \frac{n-d}{4} }  \phi_{j \ell}(\xi)  d\xi.
\]
Note that  $|\xi_i| \sim 2^{j-\ell_i}$ when  $ 1\le i \le d$, thus 
\begin{equation}\label{support}
|\mathbf{supp} \phi_{j \ell }(\xi)| \le \mathfrak{C} \prod_{i=1}^{d} 2^{j-\ell_i} ,
\end{equation}
and we have the following differential inequality:
\begin{equation}\label{differential ineq}
\left | \partial_{\xi}^{\alpha}  \left (1+|\xi|^2 \right )^{-(n-d)/4} \phi_{j \ell }(\xi)  \right|\le C_\alpha  2^{-j(n-d)/2} 2^{-(j-\ell_i)|\alpha_i|}.
\end{equation}
By performing $|\alpha_i|$ integrations by parts with respect to the variable $\xi_i$, for $1 \le i \le d$, we obtain
\begin{equation}\label{kernel size}
    \begin{array}{lr}\displaystyle
\left |\mathcal{R}_{j \ell }(x ) \right  | \le 
\left |  \int_{\R^n } e^{ 2 \pi i x  \cdot \xi   }  
\left(1+|\xi|^2 \right )^{-(n-d)/4}   \phi_{j \ell}(\xi)
d\xi  \right|
\\\\ \displaystyle
\leq C_{\alpha}
\prod_{i=1}^d |x_i|^{-|\alpha_i|} 
\left |
\int_{ \R^n}  e^{2\pi i x \cdot \xi }  \p_\xi^\alpha \left  [(1+|\xi|^2)^{-(n-d)/4} \phi_{j \ell}(\xi)  \right ]d\xi
\right |
\\\\ \displaystyle
\le  C_\alpha 
2^{-j(n-d)/2 }  \prod_{i=1}^d  \left (2^{j-\ell_i}|x_i| \right)^{-|\alpha_i| }2^{(j-\ell_i) n_i}
 \qquad \hbox{\small{Applying  (\ref{support})-(\ref{differential ineq})}}
\\\\ \displaystyle
=C_\alpha   \prod_{i=1}^d  2^{-\ell_in_i (n-d)/2n} \prod_{i =1}^{d} \left (2^{j-\ell_i}|x_i| \right )^{-|\alpha_i| }2^{(j-\ell_i)n_i (1-(n-d)/2n)}.
\end{array}
\end{equation}

Since  $n-d>0$,  we choose 
\begin{equation}\label{choose parameter}
    \begin{array}{cc}\displaystyle
|\alpha_i|=0 \quad \hbox{if } \quad |x_i|\leq 2^{-(j-\ell_i)}\quad  \quad |\alpha_i|=n_i \quad \hbox{if } \quad |x_i|>2^{-(j-\ell_i)}, \quad  1\le i \le d.
\end{array}
\end{equation}
From  (\ref{kernel size}) and  (\ref{choose parameter}), we have

\begin{equation*}
  \begin{array}{lr}\displaystyle
\left|  \mathcal{R}_{ \ell}(x) \right | \le C_\alpha
 \sum_{j\ge \ell_M  } \left\{   \prod_{i=1}^d  2^{-\ell_in_i (n-d)/2n }  \left(2^{j-\ell_i}|x_i| \right)^{-|\alpha_i|}2^{(j-\ell_i)n_i (1-(n-d)/2n)}\right\}
  \\\\ \displaystyle
\le   C_\alpha   \prod_{i=1}^d  2^{-\ell_in_i (n-d)/2n } \left\{    \sum_{j\ge \ell_M  }  \left(2^{j-\ell_i}|x_i| \right)^{-|\alpha_i| }2^{(j-\ell_i)n_i (1-(n-d)/2n)}\right\}
  \\\\ \displaystyle
 ~\leq C ~\prod_{i =1}^d  2^{-\ell_in_i (n+d)/2n }  
\left \{\sum_{|x_i|\le 2^{-j+\ell_i}}2^{ (j-\ell_i) n_i(n+d)/2n} 
+\left ({\frac{1}{|x_i|}} 
\right )^{n_i} \sum_{|x_i|>2^{-j+\ell_i}} 2^{ -(j-\ell_i) n_i(n-d)/2n} \right \}
\\\\ \displaystyle
\le C \prod_{i=1}^{d-1} 2^{-\ell_in_i (n-d)/2n}
\left ({\frac{1}{|x_i|}}\right )^{n_i\left (1-{\frac{n-d}{2n}}\right )},\quad \ell_d=0.
\end{array}  
\end{equation*}
For each subspace, applying the Hardy-Littlewood-Sobolev inequality and the Minkowski integral inequality, we have
\[
\begin{array}{lc}\displaystyle 
\left \{ \int_{\R^n} 
\left | \mathcal{T}_{\ell} f(x) \right |^2 d x \right \}^{\frac{1}{2}}
\\\\ \displaystyle
\le \mathfrak{C}
\prod_{i =1}^{d-1} 2^{-\ell_in_i (n-d)/2n}
\left \{\int_{\R^n}
\left\{ 
\int_{\R^n } 
|f(y)|  \prod_{i =1}^{d}|x_i-y_i|^{-n_i \left (1-\frac{n-d}{2n} \right )}d y  \right \} ^2 d x  \right \}^{ \frac{1}{2}}
 \\\\ \displaystyle
   \le \mathfrak{C} 
\prod_{i =1}^{d-1} 2^{-\ell_in_i (n-d)/2n}
  \left\|f  \right\|_{L^{p}  (\R^n) },
   \quad \textit{where} \quad  \frac{1}{p}=\frac{1}{2} + \frac{n-d }{2n}.
\end{array}
\]
Therefore, we have proved (\ref{L^p to L^2}).

 
 Now consider (\ref{L^2 to L^q}). We rewrite $T^*_\ell f$ as
  \[
T^*_\ell f (x) =  \int_{\R^{n}} e^{ 2\pi \i x \cdot \xi  }  \mathcal{T}^*_\ell f(\xi) d\xi, 
\]
where
\[
\mathcal{T}^*_\ell f(\xi) =\int_{\R^{n}}  e^{-2 \pi \i \Phi(y, \xi) } \overline{ \sigma}(y, \xi)  \overline{\delta}_\ell(\xi) f(y)d y.
\]
By applying the Plancherel theorem, we only need to prove that
 \[
  \left\|  \mathcal{T}^*_\ell f \right\|_{L^2(\R^n)}
   \le C \prod_{i =1}^{d-1}  2^{-\ell_i n_i \frac{n-d}{2n} }  \left\|  f \right\|_{L^p(\R^n)},
   \quad  \frac{1}{p}=\frac{1}{2} + \frac{n-d }{2n}.
 \]
Note that, according to Hölder's inequality, for $\frac{1}{p} + \frac{1}{q} = 1$,
\[
\begin{array}{lc}\displaystyle
\left\|\mathcal{T}^*_\ell f \right \|^2_{L^2(\R^{n})} = \left<\mathcal{T}^*_\ell f, \mathcal{T}^*_\ell f \right>
= \left<\mathcal{T}_\ell \mathcal{T}^*_\ell f,  f \right> 
 \le \left\|\mathcal{T}_\ell \mathcal{T}^*_\ell f \right\|_{L^{q}(\R^{n}) }
 \left\| f \right\|_{L^p(\R^{n})}. 
 \end{array}
\]
Therefore, the inequality (\ref{L^2 to L^q}) can be simplified to
\begin{equation}\label{lemma reduction}
\begin{array}{lc}\displaystyle
 \left \| \mathcal{T}_\ell \mathcal{T}^*_\ell f  \right \|_{L^{q}(\R^{n}) }
   \le C
\prod_{i =1 }^{d-1} 2^{-\ell_i n_i \frac{n-d}{n}}
  \left\|f  \right\|_{L^p(\R^n) } ,   \quad  \frac{1}{p}=\frac{1}{2} + \frac{n-d }{2n}.
   \end{array}
\end{equation}
Now let us prove the inequality (\ref{lemma reduction}). We represent $\mathcal{T}_\ell \mathcal{T}^*_\ell$ using its kernel.
\[
\mathcal{T}_\ell \mathcal{T}^*_\ell f (x)=\int_{\R^{n}} f(y) G_\ell(x, y) d y,
\]
where  the kernel   $G_\ell(x, y )$ is 
\[
\begin{array} {lc} \displaystyle
G_\ell (x, y )= \int_{\R^{n}} e^{ 2\pi \i (\Phi(x, \xi) - \Phi(y, \xi)) } \overline{\sigma}(y,\xi) \overline{\delta}_\ell(\xi) \sigma(x, \xi) \delta_\ell(\xi) d \xi.
\end{array}
\]
And define 
\[
G_{j \ell }(x, y )= \int_{\R^{n}} e^{ 2\pi \i (\Phi(x, \xi) - \Phi(y, \xi)) } \overline{\sigma}(y,\xi) \overline{\delta}_\ell(\xi) \sigma(x, \xi) \phi_{j \ell}(\xi) d \xi.
\]
Similar to the one-parameter Fourier integral operator, we can assume that the symbol function $\sigma(y, \xi)$ has a sufficiently small support in the $y$-space. Consequently, $S_\ell$ can be written as a finite sum of operators, each with a symbol function whose support is sufficiently small. Next, we observe that
\[
\nabla_{\xi_i} \left [\Phi_i(x_i, \xi_i) -\Phi_i(y_i, \xi_i) \right  ]= \nabla_{x_i}\nabla_{\xi_i} \Phi(x_i, \xi_i) \cdot (x_i -y_i)+ O(|x_i -y_i |^2 ), \quad 1\le i \le d.
\]
Since each $\Phi_i$ satisfies the non-degeneracy condition, we obtain the following estimate:
\begin{equation}\label{Phi_i est}
\left |\nabla_{\xi_i}  \Phi_i(x_i, \xi_i) - \nabla_{\xi_i} \Phi_i(y_i, \xi_i)  \right  | \ge C |x_i -y_i |,  \quad 1\le i \le d.
\end{equation}
Given the condition $\sigma \in \S^{-(n-d)/2}(\R^{\vn})$, the following differential inequality holds,
\[
\left | \p_\xi^\alpha   \left[ \overline{\sigma}(y,\xi) \overline{\delta}_\ell(\xi) \sigma(x, \xi) \phi_{j \ell}(\xi)  \right] \right | \le
 C_\alpha 2^{-j(n-d)} \prod_{i =1}^d  2^{ -( j-\ell_i)|\alpha_i|}.
\]
By performing integration by parts with respect to $\xi$, and applying the above differential inequality along with (\ref{Phi_i est}), we obtain
\[
\begin{array}{lc}\displaystyle 
\left |G_{j \ell }  (x,  y ) \right | \le   C_\alpha 2^{-j(n-d)} \prod_{i =1}^d \left  (2^{j-\ell_i} |x_i -y_i | \right )^{-|\alpha_i|} 2^{(j-\ell_i)n_i}
\\\\ \displaystyle
 \le   C_\alpha \prod_{i=1}^{d}
  2 ^{-\ell_i n_i\frac{n-d}{n}} \prod_{i =1}^d  \left  (2^{j-\ell_i} |x_i -y_i | \right )^{-|\alpha_i|} 2^{(j-\ell_i)n_i \left(1 - \frac{ n-d}{n}\right)}.
\end{array}
\]
We choose 
\begin{equation*}
    \begin{array}{cc}\displaystyle
|\alpha_i|=0 \quad \text{if} \quad |x_i|\leq 2^{-(j-\ell_i)}\quad  \quad |\alpha_i|=n_i  \quad \text{if} \quad  |x_i|>2^{-(j-\ell_i)}, \quad  1\le i \le d.
\end{array}
\end{equation*}
Thus, we can obtain
\begin{equation*}
  \begin{array}{lr}\displaystyle
\left|  G_{ \ell}(x, y) \right | \le C_\alpha
 \sum_{j\ge \ell_M  } \left\{   \prod_{i=1}^d  2^{-\ell_in_i (n-d)/n }  \left(2^{j-\ell_i}|x_i - y_i| \right)^{-|\alpha_i|}2^{(j-\ell_i)n_i (1-(n-d)/n)}\right\}
  \\\\ \displaystyle
\le   C_\alpha   \prod_{i=1}^d  2^{-\ell_in_i (n-d)/n } \left\{    \sum_{j\ge \ell_M  }  \left(2^{j-\ell_i}|x_i| \right)^{-|\alpha_i| }2^{(j-\ell_i)n_i (1-(n-d)/n)}\right\}
  \\\\ \displaystyle
 ~\leq C ~\prod_{i =1}^d  2^{-\ell_in_i (n-d)/n }  
\left \{\sum_{|x_i|\le 2^{-j+\ell_i}}2^{ (j-\ell_i) n_i(n-d)/n} 
+\left ({\frac{1}{|x_i|}} 
\right )^{n_i} \sum_{|x_i|>2^{-j+\ell_i}} 2^{ -(j-\ell_i) n_id/n} \right \}
\\\\ \displaystyle
\le C \prod_{i=1}^{d-1} 2^{-\ell_in_i (n-d)/n}
\left ({\frac{1}{|x_i|}}\right )^{n_i\left (1-{\frac{n-d}{n}}\right )},\quad \ell_d=0.
\end{array}  
\end{equation*}
By applying the Hardy-Littlewood-Sobolev inequality, we have
\[
\left\| \mathcal{T}_\ell \mathcal{T}^*_\ell f  \right\|_{L^r(\R^{n})} \le C \prod_{i =1}^{d-1} 2 ^{-\ell_i n_i\frac{n-d}{n}}    \left\|f  \right\|_{L^{p}(\R^n) }, \quad \frac{1}{r} =\frac{1}{p} -\frac{n-d}{n},
\]
and  $\frac{1}{p} = \frac{1}{2} + \frac{n-d}{2n}$ implies $r = q$, both (\ref{lemma reduction}) and (\ref{L^2 to L^q}) hold.

\section{{Seeger-Sogge-Stein decomposition for $\S^{m}(\R^{\vn})$}}\label{multi-parameter Seeger-Sogge-Stein decomposition}
 Similar to the one-parameter case, we can construct the Seeger-Sogge-Stein decomposition on each subspace $\mathbb{R}^{n_i}$. Specifically, for convenience of notation, we let
\[
j_i := j- \ell_i, \quad 1\le i \le d.
\]
Since $j \ge \ell_M$, it follows that $j_i \ge 0$. Consider a set of points $\{\xi_{j_i}^{\nu_i}\}_{\nu_i}$ uniformly distributed on the unit sphere $\mathbb{S}^{n_i-1}$, with a grid spacing of $2^{-j_i/2}$ multiplied by an appropriate constant. Then, for any given $\xi_i \in \mathbb{R}^{n_i}$, there exists a point $\xi_{j_i}^{\nu_i}$ such that 
\[
\left | \frac{\xi_i}{|\xi_i|}  -\xi^{\nu_i}_{j_i}\right |\le 2^{-\frac{j_i }{2}}.
\]
On the other hand, the number of elements in the set $\{ \xi_{j_i}^{\nu_i} \}_{\nu_i}$ is at most a constant multiple of $2^{j_i \frac{n_i-1}{2}}$. Referring to the definition of $\varphi$ in (\ref{varphi}), we define
\begin{equation*}
\varphi_{j_i}^{\nu_i} (\xi_i) =\varphi\left (  2^{j_i/2}  \left | \frac{\xi_i}{|\xi_i|}-\xi^{\nu_i}_{j_i}\right|  \right ),
\end{equation*}
Its support is contained within
\[
\Gamma_{j_i}^{\nu_i}=\left\{ \xi_i\in \mathbb{R}^{n_i}: \left| \frac{\xi_i}{|\xi_i|}-\xi^{\nu_i}_{j_i}\right |  \le 2\cdot 2^{-j_i/2}  \right \}.
\]
We also define
\begin{equation}\label{chi definition} 
\begin{array}{lc}\displaystyle
\chi_{j_i}^{\nu_i}(\xi_i)= \frac{\varphi_{j_i}^{\nu_i}(\xi_i)}{\sum_{\nu_i } \varphi_{j_i}^{\nu_i} (\xi_i)},  
\\\\  \displaystyle 
\chi_{j \ell }^\nu(\xi )= \prod_{i=1}^d  \chi_{j_i}^{\nu_i}(\xi_i), \quad  \nu =(\nu_1, \nu_2, \dots, \nu_d).
\end{array}
\end{equation}
Thus, we have a partition of unity,
\[
\sum_{\nu} \chi_{j \ell }^\nu(\xi ) =1, \quad \xi_i \neq 0, \quad 1 \le i \le d.
\]
Note that the sum above contains at most $C \prod_{i=1}^d 2^{-j_i(n_i-1)/2}$ terms.

By performing a rotation transformation, we can align the direction of $\xi_{i1}$ with $\xi_i^{\nu_i}$ in the space $\xi_i = (\xi_{i1}, \xi_i')$, where $\xi_i'$ is perpendicular to $\xi_{j_i}^{\nu_i}$.
Thus, for multi-index $\alpha_i$, we have
\begin{equation}\label{chi estimate 1}
|\partial_{\xi_{i1}}^{\alpha_i} \chi_{j_i}^{\nu_i}(\xi_i)|\le C_{\alpha_i}|\xi_i|^{-|\alpha_i|}, 
\end{equation}
\begin{equation}\label{chi estimate 2}
|\partial_{\xi'_{i}}^{\beta_i} \chi_{j_i}^{\nu_i} (\xi_i)|\le  C_{\beta_i}  2^{|\beta_i| j_i/2} |\xi_i|^{-|\beta_i|}.
\end{equation}
Now we can define the \emph{region of  influence } in the multi-parameter setting. For each $i$, we first define the rectangle $R_{j_i}^{\nu_i}$ as the set of all $x_i$ that satisfy the following conditions,
\[
\left | \left <\overline{y}_i-\nabla_{\xi_i} \Phi_i \left(x_i, \xi_{j _i}^{\nu_i} \right ), \xi_{j _i}^{\nu_i}  \right > \right |\le C 2^{-j_i},
\] 
 and 
\[
\left | \overline{y}_i-\nabla_{\xi_i} \Phi \left (x_i, \xi_{j_i}^{\nu_i} \right ) \right |\le C 2^{-j_i/2}.
\] 
Then, $R_{j_i}^{\nu_i}$ is approximately a rectangle where one side has length $C 2^{-j_i}$ and the remaining $n_i-1$ sides have length $C  2^{-j_i/2}$.

Since $\Phi_i$ satisfies the non-degeneracy condition, the mapping
\[
x_i \mapsto y_i =\nabla_{\xi_i} \Phi_i\left (x_i, \xi_{j_i}^{\nu_i} \right )
\]
For each $\xi_{j_i}^{\nu_i}$ with a non-zero Jacobian, we have
\[
\left |R_{j_i}^{\nu_i} \right |\le \mathfrak{C} 2^{-j_i}  \cdot 2^{-j_i(n_i-1)/2} .
\]


\subsection{Region of influence}
Now, using the atomic decomposition of $H^1(\mathbb{R}^n)$, suppose $a(x)$ is an atom in $H^1(\mathbb{R}^n)$ supported in a ball $B_r(\overline{y})$ with center $\overline{y}$ and radius $r$. We can assume $r < 1$, because when $r > 1$, 
\begin{equation*}
\left \|T a \right\|_{L^1(\R^n)}\le  \left \| T  a \right \|_{L^2(\R^n)}\le C   \left\| a \right \|_{L^2(\R^n)}\le  C  \left |B \right |^{-1/2}\le  C.
\end{equation*}
The first inequality holds because $\sigma(x, \xi)$ has compact support in $x$ and by using the Cauchy-Schwarz inequality. The second inequality holds because $T$ is bounded on $L^2(\mathbb{R}^n)$.

Let $k \ge 0$ be an integer such that the following inequality holds,
\[
2^{-k}\le r \le 2^{-k+1}.
\]
The \textbf{influence region} of a multi-parameter Fourier integral operator is defined as
\begin{equation}
Q= \bigotimes_{i =1}^d  Q_i , ~~~~~ Q_i=\bigcup_{j_i \ge  k } \bigcup_{\nu_i }R_{j_i}^{\nu_i}, \quad  1 \le  i \le d.
\end{equation}
Therefore, we have $Q \subset \mathbb{R}^{n}$, and elementary calculations show that
\begin{equation}
|Q|\le C \prod_{i=1}^d  \sum_{j_i \ge k } |R_{j_i}^{\nu_i} | \cdot 2^{j_i(n_i-1)/2} 
\le C \prod_{i=1}^d  \sum_{j_i \ge k } 2^{-j_i} \le C r^d.
\end{equation}
Now, by applying the Cauchy-Schwarz inequality and the estimate (\ref{L^p to L^2}) from Lemma \ref{L^2 estimates lemma}, we have
\begin{equation*}
\begin{array}{lc}\displaystyle
\int_{Q} \left  | T_{\ell} a (x) \right |  dx
\le  C |Q|^{\frac{1}{2}}
  \left  \{  \int_{\R^{n}} \left  |  T_{\ell} a(x) \right |^2 d x \right  \} ^{\frac{1}{2} }
\\\\ \displaystyle
\le C  |Q|^{\frac{1}{2}} \prod_{i =1}^{d-1} 2^{-\ell_i n_i \frac{n-d}{2n}}
  \left\| a\right\|_{L^{p} (\R^n)}
   \\\\ \displaystyle
\le C  r^{d/2}\prod_{i=1}^{d-1}  2^{-\ell_i n_i \frac{n-d}{2n}}
r^{(-1+1/p) n},\qquad \frac{1}{p} =\frac{1}{2} + \frac{n-d }{2n}
     \\\\ \displaystyle
\le C  \prod_{i=1}^{d-1} 2^{-\ell_i n_i \frac{n-d}{2n}}.
   \end{array}
\end{equation*}
Therefore, we have 
\begin{equation}\label{inside the region of influence}
\int_{Q} \left  | T_{\ell} a(x) \right |  dx
 \le C  \prod_{i=1}^{d-1} 2^{-\ell_i n_i \frac{n-d}{2n}}.
\end{equation}

 \section{Majorization of the kernels }\label{majorization on hardy space}

 Referring to the definition of $\chi_{j \ell }^\nu(\xi)$ in (\ref{chi definition}), we define the partial operator $T_{j \ell }^\nu$ as 
\begin{equation}\label{partial operator 2}
T_{j \ell }^\nu f(x)=\int_{\mathbb{R}^n} f(y)K^\nu_{j \ell}(x, y)dy,
\end{equation}
where
\begin{equation}\label{partial kernel 2}
K^\nu_{j \ell  }(x, y)=\int_{\mathbb{R}^n} e^{2\pi i  \left ( \Phi(x,\xi)-y\cdot \xi \right ) }\sigma(x,\xi) \phi_{j \ell }(\xi) \chi^\nu_{j \ell  }(\xi )d\xi.
\end{equation}
Our goal is to prove the following key estimate concerning the kernel.
\begin{lemma}\label{majorization}
Suppose  $\sigma(x,\xi) \in \S^{ -(n-d)/2}(\R^\vn ) $, then
\begin{equation}\label{estmate 1}
\int_{\mathbb{R}^n} \left | K_{j \ell }(x, y) \right |dx\le C   \prod_{i =1}^{d-1} 2^{-\ell_i(n_i-1)/2}.
\end{equation}
moreover, if  $y ,y'\in \R^n$, then 
\begin{equation}\label{estmate 2}
\int_{\mathbb{R}^n} \left |K_{j \ell }(x, y)-K_{j \ell }(x, y' ) \right  |dx\le C  2^{j} 
 |y- y'| \cdot   \prod_{i =1}^{d-1} 2^{-\ell_i(n_i-1)/2}, 
\end{equation}
and if   $y\in B_r(\overline{y} )$, $j>k+ \ell_M$,
\begin{equation}\label{estmate 3}
\int_{Q^c}    \left  |K_{j \ell }(x, y) \right |dx 
\le C  2^{-j+k+ \ell_M} \cdot   \prod_{i =1}^{d-1} 2^{-\ell_i(n_i-1)/2}.
\end{equation}

\end{lemma}


\begin{proof}
Now for each  $i $,  We have 
\begin{equation*}
\Phi_i(x_i,\xi_i)-y_i\cdot\xi_i=\left  [\nabla_{\xi_i}\Phi_i\left (x_i, \xi_{j_i}^{\nu_i}  \right)-y_i \right ]\cdot \xi_i 
+\Psi_i(x_i,\xi_i),
\end{equation*}
where 
\begin{equation*}
\Psi_i(x_i,\xi_i)=\Phi_i(x_i,\xi_i) -\nabla_{\xi_i}\Phi_i\left(x_i,  \xi_{j_i}^{\nu_i} \right )\cdot \xi_i,
\end{equation*}
Similar to the one-parameter Fourier integral operator, for any multi-indices $\alpha_i$ and $\beta_i$,
\begin{equation}\label{Psi estimate}
\left |\partial_{\xi_{i 1}}^{\alpha_i} \Psi_i(x_i,\xi_i) \right |\le C_{\alpha_i}\cdot 2^{-|\alpha_i|j_i},~~\left |\partial_{\xi'_{i}}^{\beta_i} \Psi_i(x_i,\xi_i) \right |\le C_{\beta_i}\cdot 2^{-|\beta_i|  j_i/2}.
\end{equation}
We rewrite the kernel as
\begin{equation*}
K^\nu_{j \ell   }(x, y)=\int_{\mathbb{R}^n} 
\prod_{i =1}^d  e^{2 \pi i \left   [\nabla_{\xi_i}\Phi_i \left (x_i, \xi_{j _i}^{\nu_i} \right )-y_i \right ]\cdot \xi_i  }
\Theta^\nu_{j \ell}(x,\xi) d\xi.
\end{equation*}
where
\begin{equation}\label{multi Theta}
\Theta^\nu_{j \ell  }(x,\xi)=\prod_{i =1}^d   e^{2\pi i \Psi_i(x_i,\xi_i) } \sigma(x,\xi) \phi_{j \ell }(\xi)  \chi^\nu_{j \ell  }(\xi).
\end{equation}
Next, we introduce the differential operator
\[
L=\prod_{i =1}^d   (I-2^{j_i}\nabla_{\xi_{i1}} \cdot 2^{j_i}\nabla_{\xi_{i1}}  )(I-2^{j_i/2}\nabla_{\xi'_i} \cdot 2^{j_i/2}\nabla_{\xi'_i}  ).
\]
Since we already have estimates from (\ref{chi definition}) to (\ref{chi estimate 2}) and (\ref{Psi estimate}) to (\ref{multi Theta}), and given the assumption $\sigma(x, \xi) \in \S^{-(n-d)/2}(\R^{\vn})$, we have
\[
\left |L^N\Theta_{j \ell }^\nu (x,\xi) \right |\le C_N   2^{-j(n-d)/2}, ~~N\ge 0.
\]
On the other hand, the measure of the support of $\Theta_{j \ell }^\nu (x,\xi)$ in the variable $\xi$ is
\[
\left |\mathbf{supp} ~\Theta_{j \ell }^\nu (x,\xi) \right |\le C \prod_{i =1}^d  2^{j_i} 2^{ j_i(n_i-1)/2}.
\]
Therefore, using integration by parts, we obtain:
\begin{equation}\label{Omega estimate 1}
\begin{array}{lc}\displaystyle
\left |K_{j \ell }^\nu (x, y) \right |\le
C \prod_{i =1}^d  2^{j_i} 2^{ j_i(n_i-1)/2}
2^{-j(n-d)/2}
 \\\\ \displaystyle
 \times 
 \left \{1+2^{j_i}|(\nabla_{\xi_i}\Phi_i\left (x_i,\xi^{\nu_i}_{j_i} \right )-y_i )_1| \right \}^{-2N}
\left \{1+2^{ j_i/2}|(\nabla_{\xi_i}\Phi_i \left (x_i,\xi^{\nu_i}_{j_i} \right )-y_i )'| \right \}^{-2N}.
\end{array}
\end{equation}
Here, $(\cdot)_1$ denotes the component along the direction $\xi_{j_i}^{\nu_i}$, and $(\cdot)'$ denotes the component perpendicular to it. To compute $\int_{\mathbb{R}^n} |K_{j \ell }^\nu (x,y)| dx$, we have
\[
\begin{array}{lc}\displaystyle
\int_{\mathbb{R}^n}
\left  |K_{j \ell }^\nu(x, y) \right | dx
  \le
C \prod_{i =1}^d  2^{ j_i} 2^{j_i(n_i-1)/2}
2^{-j(n-d)/2}
 \\\\ \displaystyle
 \times 
\int_{\mathbb{R}^{n_i}} \left \{1+2^{j_i}|(\nabla_{\xi_i}\Phi_i \left (x_i,\xi^{\nu_i}_{j_i} \right )-y_i )_1| \right \}^{-2N}
\left \{1+2^{ j_i/2}|(\nabla_{\xi_i}\Phi_i \left (x_i,\xi^{\nu_i}_{j_i} \right )-y_i )'| \right \}^{-2N}dx_i,
\end{array}
\]
We perform a variable substitution
\[
x_i \mapsto \nabla_{\xi_i}\Phi_i \left (x_i,\xi_{j_i}^{\nu_i} \right ),
\]
The Jacobian of this transformation is non-zero because each $\Phi_i\left(x_i, \xi^{\nu_i}_{j_i}\right)$ satisfies the non-degeneracy condition. Therefore, we obtain
\begin{equation}\label{Omega estimate}
\begin{array}{lc} \displaystyle
\int_{\mathbb{R}^n}|K_{j \ell }^\nu(x,y)|dx\le  
C  \prod_{i =1}^d  2^{ j_i} 2^{ j_i(n_i-1)/2}
2^{-j(n-d)/2}
\\\\ \displaystyle
\times
\int_{\mathbb{R}^{n_i}} \left \{1+2^{j_i}|(x_i-y_i )_1| \right \}^{-2N}
\left \{1+2^{j_i/2}|(x_i -y_i )'| \right \}^{-2N}dx_i
 \\\\ \displaystyle
 \le  C  2^{-j(n-d)/2},
\end{array}
\end{equation}
The final inequality holds if we choose $N$ such that $2N > n$.

Observing that the factor introduced by differentiation in the $y$ direction is controlled by $C 2^{j}$, similar estimates hold for $\nabla_{y} K_{j \ell }^\nu (x,y)$. The result is
\[
\int_{\mathbb{R}^n} \left |\nabla_{y} K_{j \ell }^\nu (x, y) \right |dx\le \mathfrak{C} 2^{j} 2^{-j(n-d)/2},
\]
Thus 
\begin{equation}\label{Omega estimate 2}
\int_{\mathbb{R}^n} \left  |K_{j \ell }^\nu (x, y)- K_{j \ell }^\nu (x, y') \right |dx\le \mathfrak{C} 2^{j} |y-y'|
2^{-j(n-d)/2}.
\end{equation}
Now, let us estimate
\[
\int_{Q^c}    \left |K_{j \ell }^\nu (x, y) \right |dx .
\]
Since $k$ is an integer satisfying $2^{-k} \le r \le 2^{-k+1}$, for each $i$, there exists a unit vector $\xi_{k}^{\mu_i}$ such that
\[
|\xi^{\nu_i}_{j_i} -\xi^{\mu_i}_{k}|\le 2^{-k/2}.
\]
Let $x \in Q^c$. Then there exists $1 \le i \le d$ such that $x_i \in Q_i^c$. Since $Q_i = \bigcup_{j_i \ge k} \bigcup_{\nu_i} R_{j_i}^{\nu_i}$, we have
\[
x_i \in  Q_i ^c  \subset \left ( R_{{k}}^{\mu_i} \right)^c.
\]
According to the definition of $R_{k}^{\mu_i}$,
\[
2^{k} \left | \left <\overline{y}_i-\nabla_{\xi_i} \Phi_i \left (x_i, \xi_k^{\mu_i} \right ), \xi_k^{\mu_i} \right > \right |\ge C, 
\] 
or 
\[
2^{k/2} \left |  \overline{y}_i-\nabla_{\xi_i} \Phi_i\left (x_i, \xi_k^{\mu_i} \right ) \right |\ge C.
\] 
When $y \in B_r(\overline{y})$ and $|y - \overline{y}| \le 2^{-k+1}$, since we can assume $C$ is sufficiently large, similarly to the single-parameter case, when $j_i \ge k$, we have
\[
2^{j_i} \left | \left <y_i-\nabla_{\xi_i} \Phi_i \left (x_i, \xi_{j_i}^{\nu_i} \right ),  \xi_{j_i}^{\nu_i }\right > \right |\ge C 2^{j_i-k},
\] 
or 
\[
2^{j_i/2} \left | \left ( y_i-\nabla_{\xi_i} \Phi_i \left (x_i, \xi_{j_i}^{\nu_i}  \right ) \right  )' \right |\ge C 2^{(j_i-k)/2}.
\] 
Substituting this bound into the inequality (\ref{Omega estimate 1}) and arguing as before, we obtain
\[
\begin{array}{lc}\displaystyle
\int_{Q^c}    \left  |K_{j \ell }^\nu (x, y) \right |dx   \le C 
2^{-j_i+k} \prod_{i =1}^d  2^{j_i}  2^{ -j_i(n_i-1)/2 } 2^{-j(n-d)/2}
 \\\\ \displaystyle
 \times 
\int_{\mathbb{R}^{n_i}} \left \{1+2^{j_i}|(x_i-y_i )_1 | \right \}^{1-2N}
\left \{1+2^{j_i/2}|(x_i-y_i )'| \right \}^{2-2N} dx_i.
\end{array}
\]
Similarly,
\begin{equation}\label{Omega estimate 3}
\int_{Q^c}      \left |K_{j \ell}^\nu (x, y) \right |dx   \le C 
2^{-j_i+k}  2^{-j(n-d)/2}
\end{equation}
Finally, summing over $\nu$ in the inequalities (\ref{Omega estimate}) to (\ref{Omega estimate 3}), and considering that the number of terms involved is at most a constant multiple of $\prod_{i=1}^d 2^{j_i(n_i-1)/2}$, this completes the proof of the required estimate.

 \end{proof}

\section{Proof of Theorem  \ref{classical Hardy space}  }

Now, let us prove Theorem \ref{classical Hardy space}. We only need to show that
\begin{equation}\label{final aim 1}
\begin{array}{lc}\displaystyle
\int_{\mathbb{R}^n} \left | T_{\ell} a(x) \right |dx\le  C  \prod_{i=1}^{d-1} 2^{- \epsilon \ell_i }, \quad \epsilon>0.
\end{array}
\end{equation}

\begin{proof}
{\bf Case 1:  $\ell_M > \frac{4}{3} k $}. 
We write the summation as
\[
\sum_{j\ge \ell_M}=\sum_{\ell_M \le j<  2 \ell_M }+\sum_{j  \ge 2\ell_M}.
\]
First, consider the case where $j \ge 2 \ell_M > \ell_i + k$. Referring to the definition of $T_{j \ell }$ in (\ref{partial operator jell}), and applying Lemma \ref{majorization}, we have 
\begin{equation*}
\begin{array}{lc}\displaystyle
\sum_{j \ge 2 \ell_M} 
 \int_{Q^c}
 \left |  T_{j \ell} a(x)  \right |dx 
\le C \sum_{j \ge 2 \ell_M}
   \int_{\mathbb{R}^n}    \left\{
 \int_{ Q^c}     \left  |K_{j \ell} (x, y) \right |dx \right \} \left  |a (y) \right  |dy
\\\\ \displaystyle
\le C \sum_{j \ge 2 \ell_M}
2^{-j+\ell_M+k}
 \prod_{i=1}^{d-1 } 2^{-\ell_i(n_i-1)/2}
 \int_{\mathbb{R}^n} |a(y) |dy
 \\\\ \displaystyle
 \le  C 
   \prod_{i =1}^{d-1} 2^{-\ell_i(n_i-1)/2}.
\end{array}
\end{equation*}
Next, consider the case where $\ell_M \le j < 2 \ell_M$. Applying Lemma \ref{majorization}, there exists an arbitrarily small $c > 0$ such that
\begin{equation*}
\begin{array}{lc}\displaystyle
\sum_{\ell_M \le j< 2 \ell_M} 
\int_{\mathbb{R}^{n} }       \left  |T_{j \ell } a(x) \right |dx 
\le  C \sum_{\ell_M \le  j< 2 \ell_M} \int_{\mathbb{R}^n} 
\left\{ \int_{\mathbb{R}^{n}}  \left |K_{j \ell }(x, y) \right  | dx   \right\} \left |a(y) \right  |dy
\\\\ \displaystyle
\le C  \ell_M    \prod_{i =1}^{d-1}  2^{-\ell_i(n_i-1)/2}
 \int_{\mathbb{R}^n} |a(y) |dy \le C    \prod_{i =1}^{d-1} 2^{-\ell_i  \left  ((n_i-1)/2-c\right  )}.
\end{array}
\end{equation*}
Combining the above two estimates and since we assume $n_i \ge 2$ for each $i$, we obtain
\begin{equation}\label{nontrivial part}
\int_{Q^c}
     \left | T_\ell a(x) \right |dx 
   \le C  \prod_{i =1}^{d-1}  2^{- \epsilon\ell_i}, \quad \epsilon>0.
\end{equation}
Combining inequality (\ref{nontrivial part}) with the previous estimate (\ref{inside the region of influence}), we obtain (\ref{final aim 1}).

{\bf Case 2: $\ell_M \le \frac{4}{3} k $}. 
At this point, with $k - \frac{3}{4} \ell_M \ge 0$, we define the influence region as
\begin{equation}
Q_\ell= \bigotimes_{i =1}^d  Q_{\ell_i} , ~~~~~ Q_{\ell_i}=\bigcup_{j_i\ge  k-  \frac{3 \ell_i}{4}  } \bigcup_{\nu_i }R_{j_i}^{\nu_i}, \quad 1\le i \le d.
\end{equation}
Therefore, we have
\begin{equation*}
\left |Q_\ell  \right |\le \mathfrak{C} \prod_{i =1}^d   \sum_{j_i \ge  k-  \frac{3\ell_i }{4}  }  \left |R_{j_i }^{\nu_i} \right | \cdot  2^{j_i(n_i-1)/2} \le  \C r^{ d} \prod_{i=1}^{d}  2^{\frac{3}{4} \ell_i}.
\end{equation*}
Now, applying the Cauchy-Schwarz inequality and using (\ref{L^p to L^2}) from Lemma \ref{L^p to L^2 lemma}, we have
\begin{equation*}
\begin{array}{lc}\displaystyle
\int_{Q_\ell} \left  | T_{\ell} a (x) \right |  dx
\le C  \left |Q_\ell \right |^{\frac{1}{2}}
  \left  \{ 
 \int_{\R^{n}}
\left  | T_{\ell} a(x) \right |^2 d x  \right  \} ^{\frac{1}{2} }
\le C  \prod_{i=1}^{d} 2^{-\ell_i n_i \frac{n-d}{2n}} 2^{\frac{3}{8}  \ell_i}.
   \end{array}
\end{equation*}
Since $n_i \ge 2$ and $n \ge 2d$, we have $n_i \frac{n-d}{2n} \ge \frac{2n-2d}{2n} \ge \frac{1}{2}$. Thus, we also have
\begin{equation}\label{inside the region of influence 2}
\begin{array}{lc}\displaystyle
\int_{Q_\ell} \left  | T_{\ell} a (x) \right |  dx
\le C \prod_{i=1}^{d-1}  2^{- \frac{1}{8} \ell_i  }.
   \end{array}
\end{equation}
Similarly, Lemma \ref{majorization} holds, with the only difference being that inequality (\ref{estmate 3}) is replaced by

\begin{lemma}\label{key est} 
Assume the symbol function $\sigma(x, \xi) \in \S^{-(n-d)/2}(\R^\vn)$. When $y \in B_r(\overline{y})$ and $j > k + \frac{\ell_M}{4}$,
\begin{equation*}
\int_{ Q_\ell^c}   \left   |K_{j \ell }(x, y) \right |dx 
\le C  2^{-j+k+ \frac{\ell_M}{4}}
 \prod_{i =1}^{d-1} 2^{-\ell_i(n_i-1)/2}.
\end{equation*}
\end{lemma}
Assuming Lemma \ref{key est} holds, we now write the summation as:
\[
\sum_{j \ge \ell_M} = \sum_{\ell_M \le  j  <   k+ \frac{1}{4} \ell_M }  +   \sum_{  j   \ge    k+ \frac{1}{4}  \ell_M  } .
\]
First, consider the case where $j \ge k + \frac{1}{4} \ell_M$. According to Lemma \ref{key est},
\begin{equation*}\label{outside the region of influence 2}
\begin{array}{lc}\displaystyle
\sum_{j  \ge k + \frac{1}{4} \ell_M   } 
 \int_{Q_\ell^c}     \left   |T_{j \ell } a(x) \right  |dx 
 \\\\ \displaystyle
\le C  \sum_{j  \ge  k+ \frac{1}{4}\ell_M}
  \int_{\mathbb{R}^n}
  \left\{  \int_{Q_\ell^c}      |K_{j \ell } (x, y)|dx  \right \} |a(y) |dy
\\\\ \displaystyle
 \le C 
\sum_{j  \ge  k+ \frac{1}{4}\ell_M}  2^{-j+k+ \frac{\ell_M}{4} }
  \prod_{i =1}^{d-1} 2^{-\ell_i(n_i-1)/2}
 \int_{\mathbb{R}^n} |a(y) |dy
 \\\\  \displaystyle
  \le C
  \prod_{i =1}^{d-1} 2^{-\ell_i(n_i-1)/2}.
\end{array}
\end{equation*}
Next, consider the case where $\ell_M \le j < k + \frac{1}{4} \ell_M$. Note that $\int a(y) dy = 0$, thus,
\begin{equation*}\label{}
\begin{array}{lc}\displaystyle
\sum_{ \ell_M \le   j  <  k+ \frac{1}{4} \ell_M  } 
 \int_{\R^n}    \left |T_{j \ell } a(x) \right |dx 
 \\\\ \displaystyle
 =\sum_{ \ell_M  \le  j <  k+ \frac{1}{4} \ell_M  } 
  \int_{ \R^n}  \left |
 \int_{ \R^n}   \left    (K_{j \ell }(x, y) - K_{j \ell}(x, \overline{y }) \right  ) a(y)dy  
  \right |dx
\\\\ \displaystyle
\le C
\sum_{ \ell_M  \le  j <  k+ \frac{1}{4} \ell_M  } 
  \int_{ \R^n}  \left\{ 
 \int_{ \R^n}   \left    |K_{j \ell }(x, y) - K_{j \ell}(x, \overline{y }) \right  |dx \right \}
  \left  |a(y)  \right |dy
\\\\ \displaystyle
\le C 
\sum_{ \ell_M  \le  j <  k+ \frac{1}{4} \ell_M  } 
 2^{j-k}  
 \cdot   \prod_{i =1}^{d-1} 2^{-\ell_i(n_i-1)/2}   \int_{ \R^n}   \left  |a(y)  \right |dy
 \\\\ \displaystyle
\le C
2^{ \frac{\ell_M }{4}}
 \prod_{i =1}^{d-1} 2^{-\ell_i(n_i-1)/2}.
\end{array}
\end{equation*}
Since $n_i \ge 2$, similarly, 
 \begin{equation}\label{nontrivial part 2}
\int_{Q_\ell^c}
     \left | T_\ell a(x) \right |dx 
   \le C \prod_{i =1}^{d-1}  2^{- \epsilon\ell_i}, \quad \epsilon>0.
\end{equation} 
By combining the estimates (\ref{inside the region of influence 2}) and (\ref{nontrivial part 2}), we have proven (\ref{final aim 1}) and Theorem \ref{classical Hardy space}.
\end{proof}

Let's now prove Lemma \ref{key est}. For convenience of notation, let us define
 \[
 k_i := k- \frac{3}{4} \ell_i \ge 0, \quad 1 \le i \le d. 
 \]
For each $i$, there exists a unit vector $\xi_{k_i}^{\mu_i}$ such that
\[
|\xi^{\nu_i}_{j_i} -\xi^{\mu_i}_{k_i}|\le 2^{-k_i/2}.
\]
Let $x \in Q_\ell^c$. Then there exists $1 \le i \le d$ such that $x_i \in Q_{\ell_i}^c$. Since $Q_{\ell_i} = \bigcup_{j_i \ge k_i} \bigcup_{\nu_i} R_{j_i}^{\nu_i}$, we have
\[
x_i \in  Q_{\ell_i} ^c  \subset \left ( R_{{k_i }}^{\mu_i} \right)^c.
\]
According to the definition of $R_{k_i}^{\mu_i}$,
\[
2^{k_i} \left | \left <\overline{y}_i-\nabla_{\xi_i} \Phi_i \left (x_i, \xi_{k_i}^{\mu_i} \right ),   \xi_{k_i}^{\mu_i}\right > \right |\ge C, 
\] 
or 
\[
2^{k_i/2} \left |  \overline{y}_i-\nabla_{\xi_i} \Phi_i\left (x_i, \xi_{k_i}^{\mu_i} \right ) \right |\ge C.
\] 
When $y \in B_r(\overline{y})$ and $|y - \overline{y}| \le 2^{-k+1}$, we have $2^{k_i} \cdot 2^{-k} \le C 2^{-3/4 \ell_i} \le C$.

Since we can assume $\C$ is sufficiently large, similar to the one-parameter case, when $j_i \ge k_i$, we obtain
\[
2^{j_i} \left | \left <y_i-\nabla_{\xi_i} \Phi_i \left (x_i, \xi_{j_i}^{\nu_i} \right ),  \xi_{j_i}^{\nu_i }\right > \right |\ge C 2^{j_i-k_i},
\] 
or
\[
2^{j_i/2} \left | \left ( y_i-\nabla_{\xi_i} \Phi_i \left (x_i, \xi_{j_i}^{\nu_i}  \right ) \right  )' \right |\ge C 2^{(j_i-k_i)/2}.
\] 
Substituting this bound into inequality (\ref{Omega estimate 1}) and arguing as before, we obtain
\[
\begin{array}{lc}\displaystyle
\int_{Q_\ell^c}    \left  |K_{j \ell }^\nu (x, y) \right |dx   \le C 
2^{-j_i+k_i } \prod_{i =1}^d  2^{j_i}  2^{ -j_i(n_i-1)/2 } 2^{-j(n-d)/2}
 \\\\ \displaystyle
 \times 
\int_{\mathbb{R}^{n_i}} \left \{1+2^{j_i}|(x_i-y_i )_1 | \right \}^{1-2N}
\left \{1+2^{j_i/2}|(x_i-y_i )'| \right \}^{2-2N} dx_i.
\end{array}
\]
Similarly, we obtain
\begin{equation*}
\int_{Q_\ell^c}      \left |K_{j \ell}^\nu (x, y) \right |dx   \le C 
2^{-j_i+k_i}  2^{-j(n-d)/2}.
\end{equation*}
The number of $\nu$ is at most $\prod_{i=1}^d 2^{j_i(n_i-1)/2}$, thus proving Lemma \ref{key est}.

\subsection{  The sharpness of theorem\ref{classical Hardy space}}

In this section, we will prove that Theorem \ref{classical Hardy space} is optimal. Specifically, we will show that there exists a Fourier integral operator $T$ of order $m$ such that when $-(n-d)/2 < m \le 0$ and $m > -(n-d) |1/p -1/2|$, $T$ is not bounded on $L^p(\R^n)$. Let us define
\[
\begin{array}{lc}\displaystyle
\Phi_i(x_i, \xi_i) = x_i \cdot \xi_i  + |\xi_i|, \quad i=1,2,\dots, d.
\\\\ \displaystyle
a(x) \in \mathcal{C}^\infty_0(\R^n), \quad a(x)\neq 0, \quad |x_i|=1, \quad 1,2,\dots,d.
\\\\ \displaystyle
 \gamma_{\alpha_i}(\xi_i)\in \mathcal{C}^\infty(\R^{n_i}),  ~
\gamma_{\alpha_i}(\xi_i) =    |\xi_i|^{-\alpha_i}, ~ \text{if} ~ |\xi_i|  \ge 1, \quad i=1,2, \dots, d.
\\\\\ \displaystyle 
\gamma_{-m}(\xi) = \prod_{i=1}^d \gamma_{-m_i}(\xi_i), \quad   m =\sum_{i=1}^d m_i, \quad -(n_i-1)< m_i \le 0, ~~i=1,2\dots, d. 
 \end{array}
\]
$\Psi(\xi)$ is a smooth homogeneous function of degree 0. When $\xi$ is large, the support of $\Psi(\xi)$ is contained within a truncated cone.
\[
\{\xi \in \R^n: |\xi_i|/2 \le |\xi| \le 2|\xi_i|, ~~|\xi|\ge 1,~~i=1,2, \dots,d \},
\]
Moreover, $\Psi(\xi)$ equals $1$ in a slightly smaller open subcone. Consequently, $a(x) \gamma_m(\xi) \Psi(\xi) \in \S^m(\R^\vn)$
and if we consider the operator 
 \[
T_m f_\alpha (x) = \int_{\R^n} e^{2\pi i(x\cdot \xi + \sum_{i=1}^d |\xi_i|)}  a(x) \gamma_{-m}(\xi) \Psi(\xi) \prod_{i=1}^d \widehat{f_{\alpha_i} }(\xi_i)d\xi, \quad \widehat{f_{\alpha_i} }(\xi_i) = \gamma_{\alpha_i}(\xi_i), 
\]
Then essentially 
\[
T_mf(x) = \prod_{i=1}^d T_{m_i} f_{\alpha_i}(x_i).
\]
Here,  $T_{m_i} f_{\alpha_i} $  denotes the one-parameter Fourier integral operator defined. When
\[
m > -(n-d) |1/p -1/2|,
\]
there exists $k$ such that 
  \[
  m_k > -(n_k-1) |1/p -1/2|,
  \]
then, according to the single-parameter case, take
  \[
  (\alpha_k-m_k-n_k/2-1/2)p=-1,
  \]
Thus, $T_m f_\alpha$ does not belong to $L^p(\R^n)$, but $f_{\alpha_k} \in L^p(\R^n)$ when $i \neq k$. Taking $\alpha_i < (n_i / p + n_i)$, we then have $f_\alpha \in L^p(\R^n)$.

 \subsection{The proof of corollary \ref{Lip} }

We denote the Littlewood-Paley operator $P_j$ as
\[
\widehat{P_j f} (\xi) = \phi_j(\xi) \widehat{f}(\xi).
\]
If $T_\sigma$ is a Fourier multiplier operator that satisfies
\[
\widehat{T_\sigma f}(\xi) = \sigma(\xi) \widehat{f}(\xi).
\]
Because the supports of $\phi_j(\xi)$ and $\phi_k(\xi)$ are disjoint when $|k - j| \geq 2$, 
\[
  \widehat{P_kT_\sigma P_j f}(\xi) = \phi_k(\xi) \sigma(\xi) \phi_j(\xi) \widehat{f}(\xi)=0.
\]
For multi-parameter Fourier integral operators, similarly, we have the following lemma: 
\begin{lemma}\label{orthogonal}
Let $T$ be the multi-parameter Fourier integral operator defined in Corollary \ref{Lip}, and let $f \in L^\infty(\R^n)$. Then
\[
P_k T^* P_j f(x) =0, \quad P_k T P_j f(x) =0,   \quad \text{if}  \quad |k-j| \ge 2.
\]

\end{lemma}

\begin{proof}
By definition 
\[
T^*P_jf(x) = T^*_j f(x) = \int_{\R^n} f(y) K^*_j(x, y) dy, \quad K^*_j(x, y) = \int_{\R^n} e^{ 2\pi i \left ( x\cdot \xi - \Phi(y, \xi)  \right)} 
\overline{\sigma}(y, \xi) \phi_j(\xi) d\xi. 
\]
Taking the Fourier transform of $T^*P_j f(x)$ gives
\[
\widehat{T^*P_j f} (\eta) = \int_{\R^n} f(y)  \widehat{K^*_j}(\eta, y) dy,\quad  \widehat{K^*_j}(\eta, y)= e^{- 2\pi i   \Phi(y, \xi)   } \overline{\sigma}(y, \xi) \phi_j(\xi)  
\]
Therefore, 
\begin{equation}\label{T^* definition}
\begin{array}{lc} \displaystyle
\widehat{P_k  T^* P_j f}(\eta) = \phi_k(\eta) \widehat{T^*P_j f} (\eta) = \int_{\R^n} f(y)  \phi_k(\eta)   \widehat{K^*_j}(\eta, y) dy
\\\\ \displaystyle
=  \int_{\R^n} f(y)    e^{ -2\pi i   \Phi(y, \eta)   } \overline{\sigma}(y, \eta) \phi_j(\eta)  \phi_k(\eta)   dy 
\end{array}
\end{equation}
When $|k - j| \geq 2$, the supports of $\phi_j(\eta)$ and $\phi_k(\eta)$ may be disjoint. Note that $\overline{\sigma}(y, \eta)$ has a compact support in $y$. Therefore, when $f \in \mathcal{S}(\R^n)$ or $f \in L^\infty(\R^n)$,
\[
|\widehat{P_k  T^* P_j f}(\eta)  |
\le  C  \left |  \phi_j(\eta)  \phi_k(\eta)  \right | \left \| f \right \|_{L^\infty(\R^n)}=0.
\]
Consider the adjoint operator $P_k T^* P_j$, and note that the Schwartz space $\mathcal{S}(\R^n)$ is dense in $L^1(\R^n)$, and the dual space of $L^1(\R^n)$ is $L^\infty(\R^n)$. In this way, we have proven Lemma \ref{orthogonal}.

\end{proof}

With Lemma \ref{orthogonal} established, we now proceed to prove Corollary \ref{Lip}.

\begin{proof}

First, according to the proof of Theorem \ref{classical Hardy space}, we know that
\[
\left \| T P_j f \right \|_{L^1(\R^n)}  =\left \| T_j f \right \|_{L^1(\R^n)} \le  C  \left\| f \right \|_{L^1(\R^n)}.
\]
Therefore, by duality,
\[
\left \|P_j T^*  f \right \|_{L^\infty (\R^n)} 
\le C   \left\| f \right \|_{L^\infty(\R^n)}, \quad j \ge 0.
\]
Using Lemma \ref{orthogonal}, when $j \geq 0$,
\[
\left \|P_j  T^*  f \right \|_{L^\infty (\R^n)} 
\le C \sum_{|j-k| < 2} \left \|P_j T^*  P_k f \right \|_{L^\infty (\R^n)} 
\le C \left \| P_j  T^* P_j  f  \right \|_{L^\infty (\R^n)} 
\le C   \left\| P_j f \right \|_{L^\infty (\R^n)}.
\]
Thus, 
 \[
\left \| T^* f \right\|_{Lip(\alpha)} \le C  \left \| f \right\|_{Lip(\alpha)} .
\]
Then, by duality, we have proven the desired conclusion.

\end{proof}

 \subsection{The proof of corollary \ref{h^1}}

\begin{proof}
Recalling the local Riesz transform, we have
\[
\widehat{R_j f}(\xi) = m(\xi) \widehat{f}(\xi), \quad m\in \S^0(\R^n),\quad  j=1,2,\dots,  n.
\]
Basic calculations yield
\[
\widehat{R_j T^* f}(\xi) = \int_{\R^n} f(y) \left \{  e^{-2 \pi i \Phi(y, \xi)} m(\xi) \overline{\sigma}(y, \xi ) \right \} dy
\]
If we let
\[
T_mf(x) = \int_{\R^n} e^{2 \pi i x \cdot \xi} m(\xi) \sigma(x, \xi) \widehat{f}(\xi) d \xi,  
\]
Then $R_j T^* f = T_m^* f$, because $m(\xi) \sigma(x, \xi) \in \S^{-(n-d)/2}(\R^n)$. According to Theorem \ref{classical Hardy space}, we have
\[
\left\| R_jT^*a   \right \|_{L^1(\R^n)}  = \left\| T_m^*a   \right \|_{L^1(\R^n)}  \le C.
\] 
Here, $a$ is an atom in $H^1(\R^n)$ and satisfies that the radius of the ball $B$ associated with $a$ is less than $1$. According to the atomic decomposition of $h^1(\R^n)$, we have
\[
\left\| R_jT^*f   \right \|_{L^1(\R^n)}  \le C  \left\| f   \right \|_{h^1(\R^n)} .
\] 
Thus, we have proved that $\left \| T^* f \right\|_{h^1(\R^n)} \le \left\| f \right\|_{h^1(\R^n)}$.

Note that when $f \in L^1(\R^n)$, $P_k f \in H^1(\R^n)$ for $k > 0$. Therefore,
\[
\begin{array}{lc} \displaystyle
\left\| R_j Tf  \right\|_{L^1(\R^n)} \le \sum_{k\ge 0}  C \left\| R_j  P_kT P_k f  \right\|_{L^1(\R^n)} 
\\\\ \displaystyle 
\le \sum_{k> 0}   C \left\|  P_kT P_k f  \right\|_{H^1(\R^n)} +   C \left\| R_j  S_0T S_0f  \right\|_{L^1(\R^n)} 
\\\\ \displaystyle
\le \sum_{k> 0}   C \left\|  T P_k f  \right\|_{L^1(\R^n)} + C \le C'.
\end{array}
\]
Thus, we have proven Corollary \ref{h^1}.
\end{proof}

 \subsection{The proof of corollary \ref{Sobolev} }

\begin{proof}
Recall that $P_j$ and $\mathfrak{F}_s$ are defined as
\[
\widehat{P_j f} (\xi) = \phi_j(\xi) \widehat{f}(\xi), \quad \widehat{ \mathfrak{F}_s f}(\xi) =(1+|\xi|^2)^{s/2} \widehat{f}(\xi).
\]
From the proof of Lemma \ref{orthogonal}, we know that when $f \in \mathcal{S}(\R^n)$
\[
P_k T P_j f(x) =0, \quad \text{if}  \quad |k-j| \ge 2.
\]
Let $a(x)$ be an atom in $H^1(\R^n)$. When $\sigma \in \S^{-(n-d)/2}(\R^n)$, applying the above fact, we can calculate
\[
\begin{array}{lc} \displaystyle
\left\| \mathfrak{F}_s T  \mathfrak{F}_{-s} a \right\|_{L^1(\R^n)} =\left\| \sum_{j, k \ge 0} \mathfrak{F}_s P_k T P_j  \mathfrak{F}_{-s} a  \right\|_{L^1(\R^n)} 
\\\\ \displaystyle
\le  \sum_{|k-j| \le 2 }  \left\| \mathfrak{F}_s P_k T P_j  \mathfrak{F}_{-s} a  \right\|_{L^1(\R^n)} \le C   \sum_{ j\ge0}  \left\|   P_j T P_j  a  \right\|_{L^1(\R^n)} 
\\\\ \displaystyle
 \le C   \sum_{ j\ge0}  \left\|    T P_j  a  \right\|_{L^1(\R^n)}  \le C   \sum_{ j\ge0}  \left\|    T_j  a  \right\|_{L^1(\R^n)} \le  C.
\end{array}
\]
The final inequality holds because Theorem \ref{classical Hardy space} is valid. Thus, we have proven
\begin{equation}\label{Hardy estimate}
\begin{array}{lc} \displaystyle
\left\| \mathfrak{F}_s T  \mathfrak{F}_{-s} f \right\|_{L^1(\R^n)}  
  \le C   \left\|   f  \right\|_{H^1(\R^n)}, \quad \sigma \in \S^{-(n-d)/2}(\R^\vn).
\end{array}
\end{equation}
When $\sigma \in \S^0(\R^n)$, using the Plancherel theorem and the fact that $T$ is bounded on $L^2(\R^n)$, we obtain
\[
\begin{array}{lc} \displaystyle
\left\| \mathfrak{F}_s T  \mathfrak{F}_{-s} f  \right\|_{L^2(\R^n)}  
\le  \sum_{|k-j| \le 2 }  \left\| \mathfrak{F}_s P_k T P_j  \mathfrak{F}_{-s} f   \right\|_{L^2(\R^n)} \le C   \sum_{ j\ge0}  \left\|   P_j T P_j  f \right\|_{L^2(\R^n)} 
\\\\ \displaystyle
\le C   \sum_{ j\ge0}  \left\|  \widehat{ P_j T P_j  f} \right\|_{L^2(\R^n)} 
 \le C   \sum_{ j\ge0}  \left\|  \widehat{  T P_j  f } \right\|_{L^2(\R^n)}  \le C   \sum_{ j\ge0}  \left\|   \widehat{  P_j  f}  \right\|_{L^2(\R^n)} \le  C \left\|    f \right\|_{L^2(\R^n)} .
\end{array}
\]
This means that
\begin{equation}\label{L^2 estimate last}
\begin{array}{lc} \displaystyle
\left\| \mathfrak{F}_s T  \mathfrak{F}_{-s} f \right\|_{L^2(\R^n)}  
  \le C   \left\|   f  \right\|_{L^2(\R^n)}, \quad \sigma \in  \S^{0}(\R^\vn).
\end{array}
\end{equation}
Combining the estimates (\ref{Hardy estimate}) and (\ref{L^2 estimate last}), and applying the complex interpolation theorem, we have
\[
\left\| \mathfrak{F}_s T  \mathfrak{F}_{-s} f \right\|_{L^p(\R^n)}  
  \le C   \left\|   f  \right\|_{L^p(\R^n)}, \quad \sigma \in \S^{m}(\R^\vn),
\]
if
\[
m \le -(n-d) \left( \frac{1}{p} -\frac{1}{2}\right), \quad 1<p\le 2.
\]
Hence, we have 
\[
\left\|  T  f \right\|_{L^p_s(\R^n)}  =
\left\| \mathfrak{F}_s T  f \right\|_{L^p(\R^n)}  
  \le C   \left\|  \mathfrak{F}_s   f  \right\|_{L^p(\R^n)} = C \left\|  f \right\|_{L^p_s(\R^n)}  , \quad \sigma \in \S^{m}(\R^\vn),
\]
if 
\[
m \le -(n-d) \left( \frac{1}{p} -\frac{1}{2}\right), \quad 1<p\le 2.
\]
By performing a similar discussion for the adjoint operator, we can obtain the case for $p > 2$. Thus, Corollary \ref{Sobolev} is proven.
\end{proof}

 \section{Adjoint operator $T^*$ }

Now, let us consider the adjoint operator $T^*$ of $T$. Similarly, define the partial operator $T^{\nu}_{j \ell}$ as
  \[ 
T_{j \ell}^{*\nu}f(y)  = \int_{\R^n}  f(y) K_{j \ell}^{* \nu}(x, y) dy, \quad
 K_{j \ell }^{*\nu}(x, y) = \int_{\R^n} e^{2 \pi \i \left( x  - \nabla_\xi\Phi(y, \xi_j^\nu) \right) \cdot \xi }  \Theta_{j \ell}^{*\nu}(y, \xi)d \xi.
 \]
 where 
 \[
   \Theta_{j \ell}^{*\nu}(y, \xi)= e^{ 2 \pi i  \left(  \nabla_\xi\Phi(y, \xi_j^\nu)\cdot \xi - \Phi(y, \xi) \right) }  \overline{\sigma} (y, \xi) \phi_{j \ell}(\xi) \chi^\nu_{j \ell  }(\xi )
 \]
Similar to $\Theta_{j \ell}^\nu$, when $\sigma \in \S^{-(n-d)/2}(\R^{\vn})$, we have
\[
\left | \p^\alpha_\xi  \Theta_{j \ell}^{*\nu}(y, \xi)   \right | \le C 2^{-j(n-d)/2}  \prod_{i=1}^d 2^{- j_i|\alpha_i| /2}.
\]
Thus, performing $|\alpha|$ partial integrations with respect to $\xi$, we obtain
 \begin{equation}
 \begin{array}{lc}\displaystyle
 \left |   K_{j \ell }^{*\nu}(x, y)  \right| \le C \prod_{i=1}^d  \left| x_i -\nabla_{\xi_i} \Phi(y, \xi_j^\nu) \right |^{-|\alpha_i |}
  \int_{\R^n} \left| \p_\xi^\alpha  \Theta_{j \ell}^{*\nu}(y, \xi) \right |d \xi
  \\\\ \displaystyle
  \le C  \prod_{i=1}^d |x_i|^{-|\alpha_i|}
 2^{ -j(n-d)/2}   2^{-j_i|\alpha_i|/2}\cdot 2^{j_i} \cdot 2^{j_i(n_i-1)/2}.
 \end{array}
 \end{equation}
Since there are at most $\C \prod_{i=1}^d 2^{j_i(n_i-1)/2}$ choices for $\nu$, we have
 \[
  \begin{array}{lc}\displaystyle
 \left |   K_{j \ell }^{*}(x, y)  \right|
  \le C \prod_{i=1}^d |x_i|^{-|\alpha_i|}
 2^{ -j(n-d)/2}   2^{-j_i|\alpha_i|/2}\cdot 2^{j_in_i}
 \\\\ \displaystyle
 \le C  2^{-j (|\alpha| -d)/2}  \prod_{i=1}^d |x_i|^{-|\alpha_i|}
  2^{\ell_i( |\alpha_i|/2 - n_i )}.
 \end{array}
 \]
Let $\U$ and $\W$ be subsets of ${1, 2, \dots, d}$ such that $\W \cup \U = {1, 2, \dots, d}$. Define
\[
Q_\W= \{ x \in \R^n: |x_i|>C 2^{\epsilon \ell_i }, \quad i \in \W;  \quad |x_i |< C 2^{\epsilon \ell_i }, \quad  i \in \U \}.
\]
When $\W = \emptyset$, we have $|Q_\W| \le C \prod_{i=1}^d 2^{\epsilon \ell_i n_i}$. Therefore, based on the $L^2(\R^n) \rightarrow L^p(\R^n)$ boundedness of $T_\ell^*$,
  \begin{equation*}
  \begin{array}{lc}\displaystyle
\int_{Q_\W }  \left | T_\ell^*a(x) \right |dx \le C |Q_\W|^{\frac{1}{2}} \left \| T_\ell^*a \right\|_{\L^2(\R^n)} \le C \prod_{i=1}^{d-1} 2^{\epsilon \ell_i n_i /2}  2^{-\ell_i n_i \frac{n-d}{2n}}\left \| a \right\|_{L^p(\R^n)}
\\\\ \displaystyle
\le C \prod_{i=1}^{d-1} 2^{ -\frac{ \ell_i n_i }{2} \cdot \left( 1-d/n- \epsilon \right)}  r^{-d/2}
\le C \prod_{i=1}^{d-1} 2^{ -\frac{ \ell_i n_i }{2} \cdot \left( 1-d/n- \epsilon \right)},  \qquad (r>1).
\end{array}
 \end{equation*} 
Choose $\gamma > 0$ sufficiently small, $|\alpha_i|$ sufficiently large for $i \in \W$, and $|\alpha_i| = 0$ for $i \in \U$. Then,
 \[
  \begin{array}{lc}\displaystyle
 \left |   K_{j \ell }^{*}(x, y)  \right| 
 \le C  2^{- j\gamma/2 } 2^{-j (|\alpha| -d-\gamma)/2 }  \prod_{i \in \W } |x_i|^{-|\alpha_i|}
  2^{\ell_i( |\alpha_i|/2 - n_i )}   \prod_{i \in \U }  2^{- \ell_i n_i } 
  \\\\ \displaystyle
   \le C  2^{- j\gamma/2 } 2^{-\ell_\W (|\alpha| -d-\gamma)/2 }  \prod_{i \in \W } |x_i|^{-|\alpha_i|}
  2^{\ell_\W( |\alpha_i|/2 - n_i )}   \prod_{i \in \U }  2^{- \ell_i n_i } 
 \end{array}
 \]
Here, $\ell_\W = \max_{i \in \W} \ell_i$. Therefore,
 \[
  \begin{array}{lc}\displaystyle
 \left |   K_{ \ell }^{*}(x, y)  \right| 
   \le C   2^{ -\ell_\W \left (  \sum_{i\in \W} (2n_i -d- \gamma )/2 \right )}  \prod_{i \in \W } |x_i|^{-|\alpha_i|}   \prod_{i \in \U }  2^{- \ell_i n_i } 
 \end{array}
 \]
 and 
 \[
   \begin{array}{lc}\displaystyle
   \int_{Q_\W}
 \left |   K_{ \ell }^{*}(x, y)  \right| dx
   \le C   2^{ -\ell_\W  \left (  \sum_{i\in \W} (2n_i -d- \gamma )/2 \right )}   \int_{Q_\W}\prod_{i \in \W } |x_i|^{-|\alpha_i|}  dx \cdot  \prod_{i \in \U }  2^{- \ell_i n_i } 
   \\\\ \displaystyle 
      \le C   2^{ -\ell_\W  \left (  \sum_{i\in \W} (2n_i -d- \gamma )/2 \right )}   \prod_{i \in \W } 2^{ \epsilon \ell_i(n_i- |\alpha_i|)}    \prod_{i \in \U }  2^{- \ell_in_i (1- \epsilon) } .
 \end{array}
 \]
If we take 
\[
  0<\epsilon <1 ,  \quad 1- \epsilon -d/n>0, \quad  |\alpha_i| ~ \text{Sufficiently large}, ~~i \in \W,
\]
 then there exists  $\delta>0$ such that 
 \[
 \int_{Q_\W }  \left | T_\ell^*a (x) \right |dx \le C \prod_{i=1}^{d-1} 2^{- \delta \ell_i }.
 \]
 Since 
 \[
 \R^n = \bigcup_{\W} Q_\W, 
 \]
Here, $\bigcup_{\W}$ denotes the union taken over all subsets of ${1, 2, \dots, d}$. Thus, we have proven that for $r > 1$,
   \[
 \int_{\R^n }  \left | T^*a (x) \right |dx \le C.
 \]
when $r<1$,

\begin{enumerate}
\item Using the inequality (\ref{L^2 to L^q}) instead of (\ref{L^p to L^2});

\item  The definition of $R_{j_i}^{*\nu_i}$ becomes
\[
R_{j_i}^{*\nu_i} = \left \{x_i \in \R^{n_i}: \left | \left <x_i-\nabla_{\xi} \Phi_i \left (\overline{y_i}, \xi_{j_i}^{\nu_i} \right ) , \xi_{j_i}^{\nu_i} \right > \right|\le C 2^{-j_i},  \left | x -\nabla_{\xi_i} \Phi_i \left (
\overline{y_i}, \xi_{j_i}^{\nu_i} \right )  \right |\le C 2^{-j_i/2}\right\}.
\] 
\end{enumerate}
 Then, by repeating the proof for $T$, we can prove the corresponding case for $T^*$.

 \end{proof}

\addcontentsline{toc}{section}{Acknowledgements}		
	
	\bibliography{references}

\begin{thebibliography}{MaTTT06}

\bibitem[Bea82]{B82}
R.~M. Beals.
\newblock {$L\sp{p}$} boundedness of {F}ourier integral operators.
\newblock {\em Mem. Amer. Math. Soc.}, 38(264), 1982.

\bibitem[Bre77]{B77}
P.~Brenner.
\newblock {$L\sb{p}-L\sb{p'}$}-estimates for {F}ourier integral operators
  related to hyperbolic equations.
\newblock {\em Math. Z.}, 152(3):273--286, 1977.

\bibitem[DH72]{DH72}
J.~J. Duistermaat and L.~H\"ormander.
\newblock Fourier integral operators. {II}.
\newblock {\em Acta Math.}, 128(3-4):183--269, 1972.

\bibitem[dVF76]{CF76}
Y.~Colin de~Verdi\`ere and M.~Frisch.
\newblock R\'{e}gularit\'{e} lipschitzienne et solutions de l'\'{e}quation des
  ondes sur une vari\'{e}t\'{e} riemannienne compacte.
\newblock {\em Ann. Sci. \'{E}cole Norm. Sup. (4)}, 9(4):539--565, 1976.

\bibitem[\'E70]{E70}
G.~I. \'Eskin.
\newblock Degenerate elliptic pseudodifferential equations of principal type.
\newblock {\em Mat. Sb. (N.S.)}, 82(124):585--628, 1970.

\bibitem[Fef71]{F71}
C.~Fefferman.
\newblock Characterizations of bounded mean oscillation.
\newblock {\em Bull. Amer. Math. Soc.}, 77:587--588, 1971.

\bibitem[Fef85]{F85}
R.~Fefferman.
\newblock Singular integrals on product {$H^p$} spaces.
\newblock {\em Rev. Mat. Iberoamericana}, 1(2):25--31, 1985.

\bibitem[Fef86]{F86}
R.~Fefferman.
\newblock Calder\'{o}n-{Z}ygmund theory for product domains: {$H^p$} spaces.
\newblock {\em Proc. Nat. Acad. Sci. U.S.A.}, 83(4):840--843, 1986.

\bibitem[Fef87]{F87}
R.~Fefferman.
\newblock Harmonic analysis on product spaces.
\newblock {\em Ann. of Math. (2)}, 126(1):109--130, 1987.

\bibitem[FP97]{FP97}
R.~Fefferman and J.~Pipher.
\newblock Amer. j. math.
\newblock {\em American Journal of Mathematics}, 119(2):337--369, 1997.

\bibitem[FP05]{FP05}
R.~Fefferman and J.~Pipher.
\newblock A covering lemma for rectangles in {${\Bbb R}^n$}.
\newblock {\em Proc. Amer. Math. Soc.}, 133(11):3235--3241, 2005.

\bibitem[FS72]{FS72}
C.~Fefferman and E.~M. Stein.
\newblock {$H\sp{p}$} spaces of several variables.
\newblock {\em Acta Math.}, 129(3-4):137--193, 1972.

\bibitem[H\"71]{H71}
L.~H\"ormander.
\newblock Fourier integral operators. {I}.
\newblock {\em Acta Math.}, 127(1-2):79--183, 1971.

\bibitem[MaTTT04]{MPTT04}
C.~Muscalu, J.~Pipher abd T.~Tao, and C.~Thiele.
\newblock Bi-parameter paraproducts.
\newblock {\em Acta Math.}, 193(2):269--296, 2004.

\bibitem[MaTTT06]{MPTT06}
C.~Muscalu, J.~Pipher abd T.~Tao, and C.~Thiele.
\newblock Multi-parameter paraproducts.
\newblock {\em Rev. Mat. Iberoam.}, 22(3):963--976, 2006.

\bibitem[Miy80]{M80}
A.~Miyachi.
\newblock On some estimates for the wave equation in {$L\sp{p}$} and
  {$H\sp{p}$}.
\newblock {\em J. Fac. Sci. Univ. Tokyo Sect. IA Math.}, 27(2):331--354, 1980.

\bibitem[MRS95]{MRS95}
D.~M\"uller, F.~Ricci, and E.~M. Stein.
\newblock Marcinkiewicz multipliers and multi-parameter structure on
  {H}eisenberg (-type) groups. {I}.
\newblock {\em Invent. Math.}, 119(2):199--233, 1995.

\bibitem[MRS96]{MRS96}
D.~M\"uller, F.~Ricci, and E.~M. Stein.
\newblock Marcinkiewicz multipliers and multi-parameter structure on heisenberg
  (-type) groups. ii.
\newblock {\em Math. Z.}, 221(2):267--291, 1996.

\bibitem[Per80]{P80}
J.~Peral.
\newblock {$L\sp{p}$} estimates for the wave equation.
\newblock {\em J. Functional Analysis}, 36(1):114--145, 1980.

\bibitem[Sog93]{Sogge93}
C.~Sogge.
\newblock {\em Fourier integrals in classical analysis}, volume 105 of {\em
  Cambridge Tracts in Mathematics}.
\newblock Cambridge University Press, Cambridge, 1993.

\bibitem[SSS91]{SSS91}
A.~Seeger, C.~Sogge, and E.~M. Stein.
\newblock Regularity properties of {F}ourier integral operators.
\newblock {\em Ann. of Math. (2)}, 134(2):231--251, 1991.

\bibitem[Ste93]{S93}
E.~M. Stein.
\newblock {\em Harmonic analysis: real-variable methods, orthogonality, and
  oscillatory integrals}, volume~43 of {\em Princeton Mathematical Series}.
\newblock Princeton University Press, Princeton, NJ, 1993.

\bibitem[Tao04]{T04}
T.~Tao.
\newblock The weak-type {$(1,1)$} of {F}ourier integral operators of order
  {$-(n-1)/2$}.
\newblock {\em J. Aust. Math. Soc.}, 76(1):1--21, 2004.

\bibitem[Wan22]{W22}
Z.~Wang.
\newblock Regularity of multi-parameter fourier integral operator, 2022.

\end{thebibliography}
	\bibliographystyle{alpha.bst}
\end{document}